\documentclass[11pt]{article}

\usepackage[compress]{natbib}
\usepackage{amsmath,amssymb,amsthm,amsfonts}
\usepackage[margin=1in]{geometry}
\usepackage{setspace}
\usepackage{subcaption}
\usepackage{bm}
\usepackage[dvipsnames]{xcolor}
\usepackage{float}
\usepackage{hyperref}
\hypersetup{colorlinks=true, citecolor=MidnightBlue, linkcolor=red!80!black, urlcolor=MidnightBlue}
\usepackage{pgfplots}
\pgfplotsset{compat=1.18}
\usepackage{url}            
\usepackage{booktabs}       
\usepackage{graphicx}
\usepackage{thm-restate}
\usepackage[export]{adjustbox}

\usepackage{microtype}
\usepackage{mathpazo} 
\usepackage[scaled]{helvet} 
\usepackage{courier} 
\normalfont
\usepackage[T1]{fontenc}

\usepackage{threeparttable}
\allowdisplaybreaks[4] 
\usepackage[shortlabels]{enumitem}

\usepackage{tikz}
\usetikzlibrary{fit}
\tikzset{%
	highlight/.style={rectangle,blend mode = multiply,draw=blue!90!black,thick,rounded corners = 0.3 mm,inner sep=0.5pt}
}

\usepackage{tcolorbox}

\newcommand{\diff}{\mathrm{d}}

\def\N{{\mathbb{N}}}

\def\R{{\mathbb{R}}}

\def\bI{{\mathbf{I}}}

\def\bU{{\mathbf{U}}}
\def\bV{{\mathbf{V}}}
\def\bW{{\mathbf{W}}}

\def\cA{{\cal A}}

\def\cF{{\cal F}}

\def\cO{{\cal O}}

\def\cT{{\cal T}}

\def\ord{\bm{\pi}}

\def\x{\bm{x}}
\def\w{\bm{w}}
\def\y{\bm{y}}
\def\z{\bm{z}}
\def\bz{{\mathbf 0}}

\usepackage{thmtools}

\declaretheoremstyle[parent=section]{definitionwithend}

\declaretheorem[style=definitionwithend]{theorem}
\declaretheorem[style=definitionwithend]{proposition}
\declaretheorem[style=definitionwithend]{definition}

\declaretheorem[style=definitionwithend]{remark}

\declaretheorem[style=definitionwithend]{lemma}
\declaretheorem[style=definitionwithend]{fact}

\usepackage{mathtools}
\usepackage{pifont}
\usepackage[nameinlink]{cleveref}
\usepackage[titletoc]{appendix}
\declaretheorem[style=definitionwithend]{condition}

\crefname{assumption}{Assumption}{Assumptions}
\crefname{lemma}{Lemma}{Lemmas}
\crefname{theorem}{Theorem}{Theorems}
\crefname{corollary}{Corollary}{Corollaries}
\crefname{proposition}{Proposition}{Propositions}
\crefname{condition}{Condition}{Conditions}
\crefname{claim}{Claim}{Claims}
\crefname{figure}{Figure}{Figures}
\crefname{remark}{Remark}{Remarks}
\crefname{section}{Section}{Sections}
\crefname{example}{Example}{Examples}
\crefname{definition}{Definition}{Definitions}
\crefname{table}{Table}{Tables}
\crefname{equation}{}{}
\crefname{enumi}{}{}
\crefname{appendix}{Appendix}{Appendices}
\Crefname{appendix}{Appendix}{Appendices}

\usetikzlibrary{arrows.meta,positioning}

\def\sA{{\mathsf{A}}}

\def\sZ{{\mathsf{Z}}}
\DeclareMathOperator{\supp}{supp}
\DeclareMathOperator{\Lip}{Lip}
\DeclareMathOperator{\zr}{zr}
\newcommand{\norm}[1]{\left\lVert #1\right\rVert}

\title{Lower Bounds for Nonconvex–P\L{} Minimax Optimization}

\date{\today}
\author{%
 Siyu Pan \qquad Jiajin Li\\[0.75em]
Sauder School of Business,\\
The University of British Columbia\\[0.5em]
\small \texttt{\{siyu.pan,jiajin.li\}@sauder.ubc.ca}
}

\begin{document}
\maketitle


\begin{abstract}
We study the deterministic first-order oracle complexity of finding
stationary points of the value function in smooth
nonconvex-Polyak-\L{}ojasiewicz (NC-P\L{}) minimax optimization. We
assume that the objective is jointly $\ell$-smooth and satisfies the
$\mu$-P\L{} condition in the dual variable, and that its value function
$\Phi(\x):=\max_{\y}f(\x;\y)$ satisfies
$\Phi(\bz)-\inf_{\x}\Phi(\x)\leq\Delta$. When
$\kappa:=\ell/\mu\gtrsim1$ and
$0<\epsilon^2\lesssim\ell\Delta$, we prove that every deterministic
first-order method requires
$\Omega(\ell\Delta\kappa/\epsilon^2)$ oracle queries in the worst case to
find $\x$ satisfying $\|\nabla\Phi(\x)\|\leq\epsilon$. This rate matches
the known upper bound in its dependence on
$(\ell,\Delta,\kappa,\epsilon)$ \citep{yang2022} and shows that the linear dependence on
$\kappa$ is unavoidable for deterministic first-order methods.
\end{abstract}



\section{Introduction}
\label{sec:introduction}
A fundamental distinction between strong convexity and the
Polyak-\L{}ojasiewicz (P\L{}) condition~\citep{polyak1963gradient} already appears in smooth
minimization. For an $\ell$-smooth and $\mu$-strongly convex objective with condition
number $\kappa:=\ell/\mu$, accelerated gradient methods attain the optimal
deterministic first-order oracle complexity
$\Theta(\sqrt{\kappa}\log(1/\delta))$ for reducing the objective gap by a
factor $\delta$~\citep{nemirovski,nesterov}. 
Over the broader class of $\ell$-smooth objectives satisfying the
$\mu$-P\L{} condition, the corresponding worst-case complexity is
$\Theta(\kappa\log(1/\delta))$, and gradient descent is optimal up to
numerical constants~\citep{karimi2016,yue2023}. 

This separation motivates the analogous question in smooth minimax optimization. We consider 
\begin{equation*}
    \min_{\x\in\R^m}\max_{\y\in\R^n} f(\x;\y),
    \tag{P}\label{eq:prob}
\end{equation*} 
where $f$ is jointly
$\ell$-smooth and satisfies the $\mu$-P\L{} condition with respect to the
dual variable. Let $\Phi(\x):=\max_{\y\in\R^n}f(\x;\y)$ denote the value
function, and measure stationarity by $\|\nabla\Phi(\x)\|$. We assume that
the initial  gap satisfies
$\Phi(\bz)-\inf_{\x\in\R^m}\Phi(\x)\leq\Delta$.

 When $f(\x;\cdot)$ is $\mu$-strongly concave for every $\x$, Problem~\eqref{eq:prob} belongs to the nonconvex--strongly-concave (NC-SC) setting. Since strong concavity implies the P\L{} condition, the NC-SC class is a subclass of the NC-P\L{} class considered here. For NC-SC problems, a lower bound of $\Omega(\ell\Delta\sqrt{\kappa}/\epsilon^2)$ is known for zero-respecting algorithms~\citep{li2021} and deterministic linear-span algorithms~\citep{zhang2021complexity}. On the upper-bound side, Minimax-PPA achieves $\mathcal{O}\bigl(\ell\Delta\sqrt{\kappa}\log^2(1/\epsilon)/\epsilon^{2}\bigr)$ oracle complexity~\citep{lin2020linear}. 
Catalyst-EG/OGDA further remove the accuracy-dependent logarithmic overhead and achieve $\mathcal{O}\bigl(\ell\Delta\sqrt{\kappa}\log^2(\kappa)/\epsilon^2\bigr)$ oracle complexity~\citep{zhang2021complexity}, thereby matching the lower bound up to a factor logarithmic in $\kappa$. Thus, up to logarithmic factors, the optimal dependence on the condition number is $\sqrt{\kappa}$.

For the broader NC-P\L{} class, the algorithmic landscape looks superficially similar:
AGDA requires $\mathcal{O}(\ell\Delta\kappa^{2}/\epsilon^{2})$ oracle queries~\citep{lin2020gradient}, whereas
Smoothed GDA finds an $\epsilon$-stationary point of the value function in
$\mathcal{O}(\ell\Delta\kappa/\epsilon^{2})$ \citep{yang2022}.  What is missing is
the matching obstruction: it remained unclear whether the linear dependence on $\kappa$
is intrinsic to the dual P\L{} condition or merely an artifact of the available
analyses, since the NC-SC results above leave room for a further $\sqrt{\kappa}$
improvement. 

We resolve this question by proving that, when $\kappa\gtrsim 1$ and
$0<\epsilon^{2}\lesssim\ell\Delta$, every deterministic first-order method requires
$\Omega(\ell\Delta\kappa/\epsilon^{2})$ oracle queries in the worst case. The bound
applies to arbitrary deterministic first-order methods---a strictly larger class than
the zero-respecting and linear-span classes covered by the NC-SC lower bounds above.
Together with the upper bound of \citet{yang2022}, it establishes the tight
deterministic first-order oracle complexity $\Theta(\ell\Delta\kappa/\epsilon^{2})$ in
the parameter regime considered here. In particular, the improvement from $\kappa^{2}$
to $\sqrt{\kappa}$ available under strong concavity stops at $\kappa$ under the dual
P\L{} condition: the separation between the two classes is real, and not an artifact of
suboptimal algorithms.



Our proof follows the chain composition strategy of \citet{li2021}.
However, existing P\L{} hard instances for smooth minimization, such as
that of \citet{yue2023}, are not directly composable with an outer
nonconvex chain. The required inner block must simultaneously reproduce a
prescribed term in the outer value function after dual maximization, hide
the gradient with respect to the successor primal coordinate until the
dual chain has been traversed, and accommodate the scale of the outer
term without losing uniform smoothness or the desired P\L{} modulus.
Existing standalone minimization hard instances do not provide these
properties simultaneously.

To meet these requirements, we construct a scalable length-$N$ dual
P\L{} chain. Each such chain delays the revelation of one primal
coordinate, while dual maximization recovers the corresponding term in
the outer value function exactly. Figure~\ref{fig:hard-instance-chain}
illustrates the resulting composition.

\begin{figure}[t]
    \centering
\includegraphics[width=0.6\textwidth]{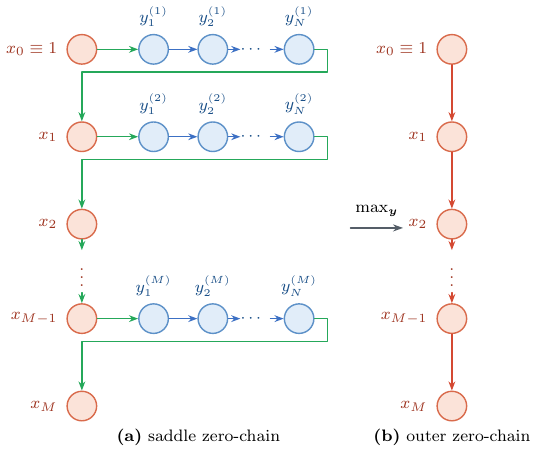}
    \caption{
    Schematic of the hard instance.
    In the saddle objective, each transition from $x_{i-1}$ to $x_i$ is
    mediated by a length-$N$ dual chain
    $(y^{(i)}_1,\ldots,y^{(i)}_N)$.
    Maximizing over the dual variables collapses these blocks and recovers
    the outer primal zero-chain as the value function.
    }
    \label{fig:hard-instance-chain}
\end{figure}

The dual blocks have smoothness bounded independently of $N$ and a
P\L{} modulus of order $1/N$, allowing us to take
$N=\Theta(\kappa)$. Inserting one block into each of the
$M=\Theta(\ell\Delta/\epsilon^2)$ links of the nonconvex zero-chain
construction of \citet{carmon2019lower} yields
$\Omega(\ell\Delta\kappa/\epsilon^2)$ sequentially hidden coordinates.
Finally, we adapt the finite-horizon resisting-oracle framework
of \citet[Lemma~7]{carmon2019lower} to the primal-dual setting; see also
\citet{li2026Weakly} for a recent formulation. This reduction extends the
lower bound from zero-respecting algorithms to arbitrary deterministic
first-order methods.

\paragraph{Notation.}
Bold lowercase letters denote vectors. Scalar coordinates are written without boldface; for example,
\(z_i\) denotes the \(i\)th coordinate of \(\z\). The semicolon in \(f(\x;\y)\)
separates the minimization and maximization variables. We take
\(\N:=\{1,2,\ldots\}\) and \(\N_0:=\N\cup\{0\}\). For \(T\in\N\),
we use the shorthand \([T]:=\{1,\ldots,T\}\) and set
\([0]:=\varnothing\). We use the convention that
\(\inf\varnothing=+\infty\). We write \(\norm{\cdot}\) for the Euclidean
norm on vectors and the induced operator norm on matrices. The zero vector
in \(\R^d\) is denoted by \(\bz_d\), or simply \(\bz\) when its dimension
is clear. The all-ones vector is denoted by \(\mathbf{1}\) whenever its
dimension is clear. We denote the \(d\times d\) identity matrix by
\(\bI_d\). For a scalar \(t\), define \(t^-:=\max\{-t,0\}\). For
\(\z\in\R^d\), define
\(\z^-:=(z_1^-,\ldots,z_d^-)\). For a differentiable function
\(F:\R^d\to\R\), we write
\(\nabla_{z_j}F(\z):=[\nabla F(\z)]_j\) for \(j\in[d]\). For a Lipschitz
function \(F\), we denote its Lipschitz constant by \(\Lip(F)\). The symbols
\(\cO\), \(\Omega\), \(\Theta\), and \(\lesssim\) hide positive numerical
constants. For any positive weight vector
\(\w=(w_1,\ldots,w_d)\in\R_{++}^d\) and \(\z\in\R^d\), define
\(\norm{\z}_{\w}^2:=\sum_{j=1}^d w_jz_j^2\)
and
\(\norm{\z}_{*,\w}^2:=\sum_{j=1}^d z_j^2/w_j\).
Then $\norm{\cdot}_{*,\w}$ is the dual norm of $\norm{\cdot}_{\w}$.
For integers \(d'\geq d\), define
\(\operatorname O(d',d)
:=
\{\mathbf U\in\R^{d'\times d}\mid
\mathbf U^\top\mathbf U=\bI_d\}\),
the set of column-orthonormal matrices in \(\R^{d'\times d}\).
A coordinate ordering is a permutation
\(\ord=(\pi_1,\ldots,\pi_d)\) of $[d]$. 
For example, the permutation \(\ord=(2,3,1)\) orders the coordinates of
\(\z=(z_1,z_2,z_3)\) as \((z_2,z_3,z_1)\).
For any \(\z\in\R^d\), define its support under \(\ord\) by
\(
    \supp_{\ord}(\z)
    :=
    \{k\in[d]\mid z_{\pi_k}\ne0\}.
\)
Under the identity permutation
\(\ord=(1,\ldots,d)\), we simply write
\(
    \supp(\z):=\{k\in[d]\mid z_k\ne0\}.
\)
Whenever a coordinate ordering is involved, we identify the pair
\((\x,\y)\) with its canonical concatenation, as in
\(\supp(\x,\y)\) and \(\supp_{\ord}(\x,\y)\).
\paragraph{Organization.}
The remainder of the paper is organized as follows.
Section~\ref{sec:pre-lb} introduces the problem class, oracle
model, and algorithm classes. Section~\ref{sec:framework} states the main
results and identifies four certificates sufficient for the lower bound.
Sections~\ref{sec:construction} and~\ref{sec:verification} construct and
verify the hard instance, respectively. Section~\ref{sec:close} concludes
the paper, and Appendix~\ref{app:ncpl-deferred-estimates} contains the
deferred technical proofs.

\section{Preliminaries}
\label{sec:pre-lb}
To study Problem~\eqref{eq:prob}, we introduce the class of NC-P\L{} functions.
\begin{definition}[NC-P\L{} function class]
\label{def:function-class}
Given $\ell,\mu,\Delta>0$, let
 $\cF(\ell,\mu,\Delta)$ be the union, over
 $(m,n)\in\N\times\N$, of the classes of functions
 $f:\R^m\times\R^n\to\R$ satisfying the following properties. For each
 such function, define $\Phi(\x):=\max_{\y\in\R^n}f(\x;\y)$.
\begin{enumerate}[(i)]
    \item {($\ell$-smoothness)} $f$ is jointly $\ell$-smooth on $\R^m\times \R^n$, i.e., for all $(\x_1,\y_1),(\x_2,\y_2)\in \R^m\times \R^n$,
    \[
        \bigl\|\nabla f(\x_1;\y_1)-\nabla f(\x_2;\y_2)\bigr\|
        \le \ell\bigl\|(\x_1,\y_1)-(\x_2,\y_2)\bigr\|.
    \]
    \item {($\mu$-dual P\L{})} For every fixed $\x\in\R^m$, 
     the problem $\max_{\y\in\R^n}f(\x;\y)$ has a nonempty solution set and a finite optimal value. Moreover,
    \[
        \frac12\norm{\nabla_{\y} f(\x;\y)}^2
        \ge\mu\left(\max_{\widetilde{\y}\in\R^n}f(\x;\widetilde{\y})-f(\x;\y)\right)
    \]
   holds for all $\x\in \R^m$ and $\y\in\R^n$. 
    \item {(Initial gap)} The initialization satisfies
        $\Phi(\bz)-\inf_{\x\in \R^m}\Phi(\x)\le \Delta$. 
\end{enumerate}
\end{definition}
The standard stationarity measure used throughout the rest of the paper is defined as follows.
\begin{definition} 
\label{def:stationary}
Given $\epsilon>0$, a point $\x\in\R^m$ is an
$\epsilon$-stationary point of Problem~\eqref{eq:prob} if
\[\|\nabla \Phi(\x)\|\leq\epsilon.\]
\end{definition} 
Definition~\ref{def:stationary} is well-defined for Problem~\eqref{eq:prob} since the value function of a smooth NC-P\L{} function is differentiable~\citep[Lemma A.5]{nouiehed2019}.


We use the first-order saddle oracle considered by
\citet{li2021}.
\begin{definition}[First-order saddle oracle]
\label{def:first-order-oracle}
For a differentiable function \(f:\R^m\times\R^n\to\R\), the first-order
saddle oracle at \((\x,\y)\in\R^m\times\R^n\) returns
\[
        \mathbb O_f(\x;\y)
        :=
        \bigl(
        f(\x;\y),
        \nabla_{\x} f(\x;\y),
        \nabla_{\y} f(\x;\y)
        \bigr).
\]
\end{definition}


We call a differentiable saddle function admissible if its
value function is finite and differentiable.
For a first-order algorithm \(\sA\) and an admissible saddle function \(f\), let \(\sA_{\x}^{(t)}[f]\) and \(\sA_{\y}^{(t)}[f]\) denote, respectively, the primal and dual components of its \(t\)-th oracle query.
We first consider zero-respecting algorithms that access \(f\) through the oracle
$\mathbb O_f$.

\begin{definition}[First-order zero-respecting algorithm]
\label{def:zero-respecting}
A first-order algorithm $\sA$ is zero-respecting if, for each admissible
$f$, it starts from
$(\sA_{\x}^{(0)}[f],\sA_{\y}^{(0)}[f])=\bz$ and its query sequence
satisfies
\[
\supp\bigl(\sA_{\x}^{(t)}[f],\sA_{\y}^{(t)}[f]\bigr)
 \subseteq
 \bigcup_{s<t}
 \supp\Bigl(
     \nabla_{\x}f\bigl(\sA_{\x}^{(s)}[f];\sA_{\y}^{(s)}[f]\bigr),
     \nabla_{\y}f\bigl(\sA_{\x}^{(s)}[f];\sA_{\y}^{(s)}[f]\bigr)
 \Bigr),
 \qquad \forall t\geq1.
\]
We denote the class of such algorithms by $\cA_{\zr}$.
\end{definition}

We next define \(\cA_{\det}\), the class of deterministic first-order algorithms.
\begin{definition}[First-order deterministic algorithm]
\label{def:deterministic-algorithm}
A first-order algorithm $\sA$ is deterministic if there exists a sequence
of maps $\{\Gamma_t\}_{t\geq1}$ such that, for every admissible $f$, it
starts from $(\sA_{\x}^{(0)}[f],\sA_{\y}^{(0)}[f])=\bz$ and its query
sequence satisfies
\[
 \bigl(\sA_{\x}^{(t)}[f],\sA_{\y}^{(t)}[f]\bigr)
 =
 \Gamma_{t}
 \left(
   \Bigl(
       \mathbb O_f\bigl(\sA_{\x}^{(s)}[f];\sA_{\y}^{(s)}[f]\bigr)
   \Bigr)_{s<t}
 \right),\qquad \forall t\geq1.
\]
We denote the class of such algorithms by $\cA_{\det}$.
\end{definition}

We measure the worst-case number of oracle queries as follows. 
\begin{definition}
\label{def:complexity-algorithm-class}
Fix $\epsilon>0$. Let $\cA\subseteq\cA_{\det}$ be any nonempty subclass of deterministic first-order algorithms, and let $\cF$ be a nonempty class of admissible saddle functions. The oracle complexity of $\cA$ over $\cF$ is defined by
\[
        \cT_\epsilon(\cA,\cF)
        :=
        \inf_{\sA\in\cA}
        \sup_{f\in\cF}
         \inf
        \left\{
        t\in\N_0\;\middle|\;
        \sA_{\x}^{(t)}[f]\;
        {\text{is an }\epsilon\text{-stationary point for }f}
        \right\}.
\]
\end{definition}
To construct the hard instance, we next extend the standard
zero-chain structure to the joint primal-dual variable.
\begin{definition}[First-order saddle zero-chain]
\label{def:saddle-zero-chain}
Let \(f:\R^m\times\R^n\to\R\) be differentiable, and let
\(\ord=(\pi_1,\ldots,\pi_{m+n})\) be a coordinate ordering of
$(\x,\y)$. We say that \(f\)
is a first-order saddle zero-chain with respect to 
\(\ord\) if, for every
\(i\in[m+n]\) and every \((\x,\y)\in\R^m\times\R^n\),
\[
    \supp_{\ord}(\x,\y)\subseteq[i-1]
    \quad\Longrightarrow\quad
    \supp_{\ord}\bigl(
        \nabla_{\x}f(\x;\y),
        \nabla_{\y}f(\x;\y)
    \bigr)
    \subseteq[i].
\]
\end{definition}
Our lower bound is first established for zero-respecting algorithms. To extend it to arbitrary deterministic first-order methods, we employ a finite-horizon resisting-oracle reduction, following the framework of \citet{carmon2019lower,li2026Weakly}. In the present saddle-oracle setting, the reduction consists of two components: separate orthogonal embeddings of the primal and dual spaces that preserve the NC-P\L{} function class and the value-function stationarity measure, and a finite-horizon transcript simulation that relates an arbitrary deterministic method on a rotated instance to a zero-respecting method on the original instance. We first establish the required rotation invariance.

For a function \(f:\R^m\times\R^n\to\R\) and matrices
\(\bU\in\operatorname O(m',m)\) and
\(\bV\in\operatorname O(n',n)\), define the rotated function
\(f_{\bU,\bV}:\R^{m'}\times\R^{n'}\to\R\) by
\begin{equation*}
    f_{\bU,\bV}(\x;\y)
    :=
    f(\bU^\top\x;\bV^\top\y).
\end{equation*}
Let
\(
    \Phi_{\bU,\bV}(\x)
    :=\max_{\y\in\R^{n'}}f_{\bU,\bV}(\x;\y)
\)
denote its value function.

\begin{lemma}[Rotation invariance]
\label{lem:isometric-invariance}
Let \(f\in\cF(\ell,\mu,\Delta)\) be defined on
\(\R^m\times\R^n\), and let
\(\bU\in\operatorname O(m',m)\) and
\(\bV\in\operatorname O(n',n)\). Then
\(f_{\bU,\bV}\in\cF(\ell,\mu,\Delta)\). Moreover, for every
\(\x\in\R^{m'}\),
\[
    \|\nabla\Phi_{\bU,\bV}(\x)\|
    =
    \|\nabla\Phi(\bU^\top\x)\|.
\]
\end{lemma}

\begin{proof}[Proof of Lemma~\ref{lem:isometric-invariance}]
Since \(\bU^\top\bU=\bI_m\) and \(\bV^\top\bV=\bI_n\), the maps
\(\bU,\bV\) are isometric embeddings, whereas
\(\bU^\top,\bV^\top\) are surjective contractions. By the chain rule,
\[
\nabla f_{\bU,\bV}(\x;\y)
=
\bigl(
\bU\nabla_{\x}f(\bU^\top\x;\bV^\top\y),
\bV\nabla_{\y}f(\bU^\top\x;\bV^\top\y)
\bigr),
\]
and hence \(f_{\bU,\bV}\) is jointly \(\ell\)-smooth.

Since \(\bV^\top\) is surjective, the original and lifted inner problems have the same optimal value. Moreover, if
\(\bar\y^\star\) maximizes \(f(\bU^\top\x;\cdot)\), then
\(\bV\bar\y^\star\) maximizes \(f_{\bU,\bV}(\x;\cdot)\), since
\(\bV^\top\bV=\bI_n\). Hence, the lifted maximum is finite and attained, and
\begin{equation}
\label{eq:value-function}
\Phi_{\bU,\bV}(\x)
=
\max_{\y\in\R^{n'}}f_{\bU,\bV}(\x;\y)
=
\max_{\bar\y\in\R^n}f(\bU^\top\x;\bar\y)
=
\Phi(\bU^\top\x).
\end{equation}
Moreover,
\[
\frac12\|\nabla_{\y}f_{\bU,\bV}(\x;\y)\|^2
=
\frac12\|\nabla_{\y}f(\bU^\top\x;\bV^\top\y)\|^2
\geq
\mu\bigl(
\Phi_{\bU,\bV}(\x)-f_{\bU,\bV}(\x;\y)
\bigr),
\]
so \(f_{\bU,\bV}\) satisfies the \(\mu\)-dual P\L{}
inequality.

Since \(\bU^\top\) is surjective,
\[
\Phi_{\bU,\bV}(\bz_{m'})
-\inf_{\x\in\R^{m'}}\Phi_{\bU,\bV}(\x)
=
\Phi(\bz_m)-\inf_{\bar\x\in\R^m}\Phi(\bar\x)
\leq\Delta.
\]
Therefore \(f_{\bU,\bV}\in\cF(\ell,\mu,\Delta)\). Finally,
\eqref{eq:value-function} gives
\(\nabla\Phi_{\bU,\bV}(\x)
=\bU\nabla\Phi(\bU^\top\x)\), and hence
\(
\|\nabla\Phi_{\bU,\bV}(\x)\|
=
\|\nabla\Phi(\bU^\top\x)\|
\)
since \(\bU\) is an isometry. This proves the lemma.
\end{proof}

We next state the finite-horizon transcript simulation used in
the resisting-oracle reduction. The proof follows the construction of
\citet[Lemma~7]{carmon2019lower}, with separate orthogonal embeddings for
the primal and dual spaces.

\begin{lemma}[Finite-horizon resisting oracle]
\label{lem:finite-horizon-resisting-oracle}
For any \(T_0\in\N\) and
\(\sA\in\cA_{\det}\), there exists a
zero-respecting deterministic algorithm
\(\sZ\in\cA_{\zr}\cap\cA_{\det}\) such that, for every
smooth saddle function \(f:\R^m\times\R^n\to\R\), one can find
\(\bU\in\operatorname O(m+T_0,m)\) and \(\bV\in\operatorname O(n+T_0,n)\) satisfying
\[
 \bigl(\sZ_{\x}^{(t)}[f],\sZ_{\y}^{(t)}[f]\bigr)
 =
 \bigl(
     \bU^\top\sA_{\x}^{(t)}[f_{\bU,\bV}],
     \bV^\top\sA_{\y}^{(t)}[f_{\bU,\bV}]
 \bigr),
 \qquad t=0,\ldots,T_0-1.
\]
\end{lemma}
Lemma~\ref{lem:finite-horizon-resisting-oracle}, together with the rotation invariance established in Lemma~\ref{lem:isometric-invariance}, yields the desired reduction. Indeed, a hard instance for zero-respecting algorithms can be rotated into an instance whose first $T_0$ oracle responses simulate those observed by any prescribed deterministic first-order method. The rotation preserves the parameters $(\ell,\mu,\Delta)$ and the norm of the value-function gradient. Consequently, any $T_0$-query lower bound for $\cA_{\zr}\cap\cA_{\det}$ transfers to the full class $\cA_{\det}$ with the same stationarity tolerance.

\section{Main Results and Hard-Instance Certificates} 
\label{sec:framework}
\subsection{Main Results}
We first state the lower bound for deterministic first-order zero-respecting algorithms.
\begin{theorem}[Zero-respecting lower bound]
\label{thm:zr-lower}
There exist numerical constants $c_0>1$ and $c_1,c_2>0$
such that, for every $\ell,\mu,\Delta,\epsilon>0$ satisfying
$\kappa:=\ell/\mu\ge c_0$ and
$\epsilon^2\le c_1\ell\Delta$, there exists
$f\in\cF(\ell,\mu,\Delta)$ such that
\begin{equation*}
    \cT_{\epsilon}(\cA_{\zr}\cap\cA_{\det},\{f\})\ge c_2\kappa\frac{\ell\Delta}{\epsilon^2}.
\end{equation*}
\end{theorem}
Our main result extends this lower bound to arbitrary deterministic first-order methods.
\begin{theorem}
\label{thm:deterministic}
There exist numerical constants $c_0>1$ and $c_1,c_2>0$
such that, for every $\ell,\mu,\Delta,\epsilon>0$ satisfying
$\kappa:=\ell/\mu\ge c_0$ and
$\epsilon^2\le c_1\ell\Delta$,
\begin{equation*}
    \cT_{\epsilon}(\cA_{\det},\cF(\ell,\mu,\Delta))\ge c_2\kappa\frac{\ell\Delta}{\epsilon^2}.
\end{equation*}
\end{theorem} 
\begin{remark}[Arbitrary deterministic output]
\label{rem:arbitrary-output}
Theorem~\ref{thm:deterministic} also covers the case of a deterministic
point computed from the final transcript but not previously queried, at
the cost of one additional oracle call. Indeed, augment the algorithm by querying
that primal point together with the dual origin. The lower bound for query
points then applies to the augmented deterministic algorithm, and the
single additional query is absorbed into the numerical constant.
\end{remark}

\subsection{Hard-Instance Certificates}
We next isolate four properties of an unscaled hard instance that suffice to prove Theorem~\ref{thm:zr-lower}. 
Call $(M,N)$ admissible if $M\geq1$ and $N\geq2$. For each
admissible pair, consider
$
\bar f:\R^M\times\R^{MN}\to\R
$
with coordinate ordering $\ord$, and denote its value function by $\bar\Phi$.
We impose four conditions on $(\bar f,\ord)$.

\begin{condition}[Unscaled hard instance properties]
\label{cond:normalized-hard-instance}
There exist numerical constants
$\ell_0,\mu_0,g_0,\Delta_0>0$, independent of $M$ and $N$, such that,
for every admissible pair $(M,N)$, the corresponding unscaled hard instance
$\bar f$ and coordinate ordering $\ord$ satisfy the following properties:
\begin{enumerate}[
    label=\textup{(C\arabic*)},
    ref=\textup{(C\arabic*)},
    leftmargin=*,
    itemsep=0.6em
]
\item
The function \(\bar f\) is jointly \(\ell_0\)-smooth and $\frac{\mu_0}{N}$-dual P\L{}.

\item
The function \(\bar f\) is a first-order saddle zero-chain with
respect to \(\ord\).

\item
The $\pi_{M(N+1)}$-th coordinate of $(\x,\y)$ is \(x_M\). Moreover,
\[
    x_M=0
    \quad\Longrightarrow\quad
    \bigl\|\nabla\bar\Phi(\x)\bigr\|
    \ge g_0.
\]

\item
The initial value-function gap satisfies
\[
    \bar\Phi(\bz)
    -
    \inf_{\x\in\R^M}\bar\Phi(\x)
    \le \Delta_0 M.
\]
\end{enumerate}
\end{condition}
The four conditions play complementary roles.
Conditions~\ref{cond:normalized-hard-instance}~\textup{(C2)}
and~\textup{(C3)} provide the information-theoretic obstruction: the
former limits coordinate discovery, while the latter rules out small
value-function gradients before the terminal coordinate is activated.
Conditions~\textup{(C1)} and~\textup{(C4)} provide the regularity and
initial-gap bounds needed to scale the unscaled instance into the target
function class.

\section{The Construction of Our Hard Instance}
\label{sec:construction}

We now construct an unscaled hard instance satisfying the four
certificates in Condition~\ref{cond:normalized-hard-instance}. Fix
$M\geq1$ and $N\geq2$. The construction separates the geometry of the
value function from the coordinate-discovery behavior of the saddle
oracle. Following the chain-composition strategy of \citet{li2021}, we
couple an $M$-stage primal zero-chain, based on the construction of
\citet{carmon2019lower}, with $M$ dual P\L{} chains of length $N$. Dual
maximization preserves the primal zero-chain as the value function, while
each dual chain delays the revelation of the next primal coordinate. A
perspective transformation adapts the dual chains to the scale of the
corresponding primal interactions without changing their P\L{} modulus.

Figure~\ref{fig:hard-instance-chain} illustrates the resulting
coordinate-discovery chain and the associated value function chain.

We let $\ord$ be the coordinate ordering induced by
$$
\bigl(
y^{(1)}_1,\ldots,y^{(1)}_N,x_1,\ldots,
y^{(M)}_1,\ldots,y^{(M)}_N,x_M
\bigr).
$$
The saddle zero-chain property forces these coordinates to be revealed
sequentially, giving a total chain length of $M(N+1)$.

We construct the dual component in two stages. First, we build
a weighted dual P\L{} chain that enforces sequential coordinate discovery.
Second, we introduce a terminal-gated primal-dual coupling and apply a
perspective transformation to match the scale of the outer construction.
\subsection{A Weighted Dual P\L{} Chain}
\label{sec:inner-dual-chain}
For $\y=(y_1,\ldots,y_N)\in\R^N$, we adopt the convention $y_0:=1$. We first define the scalar transition functions
$$
q(t):=
\begin{cases}
0, & t\leq 0,\\
3t^2-2t^3, & 0\leq t\leq 1,\\
1, & t\geq 1,
\end{cases}
\qquad\text{and}\qquad
p(t):=q\left(\frac{t+1}{2}\right).
$$
Direct differentiation gives
\[
q'(t)
=
\begin{cases}
6t(1-t), & 0<t<1,\\
0, & \text{otherwise},
\end{cases}
\qquad\mbox{and}\qquad
p'(t)
=
\begin{cases}
\dfrac34(1-t^2), & -1<t<1,\\
0, & \text{otherwise}.
\end{cases}
\]
\begin{figure}[H]
\centering
\begin{tikzpicture}[line cap=round,line join=round]

\begin{scope}[shift={(0,0)}]
\begin{scope}[xscale=1.65,yscale=2]

\draw[->,thin] (-1.6,0)--(1.65,0) node[right] {$t$};
\draw[->,thin] (0,-0.12)--(0,1.25)
    node[above] {$q(t),\,p(t)$};

\draw[very thick,blue!70!black]
    (-1.55,0)--(0,0);

\draw[
    very thick,
    blue!70!black,
    domain=0:1,
    samples=80,
    variable=\xx
]
    plot ({\xx},{3*(\xx)*(\xx)-2*(\xx)*(\xx)*(\xx)});

\draw[very thick,blue!70!black]
    (1,1)--(1.55,1);

\draw[very thick,dashed,red!75!black]
    (-1.55,0)--(-1,0);

\draw[
    very thick,
    dashed,
    red!75!black,
    domain=-1:1,
    samples=100,
    variable=\xx
]
    plot
    ({\xx},
     {3*((\xx+1)/2)*((\xx+1)/2)
      -2*((\xx+1)/2)*((\xx+1)/2)*((\xx+1)/2)});

\draw[very thick,dashed,red!75!black]
    (1,1)--(1.55,1);

\fill[red!75!black] (-1,0) circle (1.2pt);
\node[below left=2pt,font=\scriptsize]
    at (-1,0) {$(-1,0)$};

\fill[blue!70!black] (0,0) circle (1.2pt);
\node[below right=2pt,font=\scriptsize]
    at (0,0) {$(0,0)$};

\fill[red!75!black] (0,0.5) circle (1.2pt);
\node[right=3pt,font=\scriptsize]
    at (0,0.5) {$\left(0,\frac12\right)$};

\fill[black] (1,1) circle (1.2pt);
\node[above right=2pt,font=\scriptsize]
    at (1,1) {$(1,1)$};

\draw[very thick,blue!70!black]
    (-1.45,1.13)--(-1.15,1.13);
\node[right=2pt,font=\small]
    at (-1.15,1.13) {$q$};

\draw[very thick,dashed,red!75!black]
    (-0.65,1.13)--(-0.35,1.13);
\node[right=2pt,font=\small]
    at (-0.35,1.13) {$p$};

\end{scope}

\node[font=\small] at (0,-0.9)
    {(a) The functions \(q\) and \(p\)};
\end{scope}

\begin{scope}[shift={(6.8,0)}]
\begin{scope}[xscale=1.65,yscale=1.4]

\draw[->,thin] (-1.6,0)--(1.65,0) node[right] {$t$};
\draw[->,thin] (0,-0.17)--(0,1.75)
    node[above] {$q'(t),\,p'(t)$};

\draw[very thick,blue!70!black]
    (-1.55,0)--(0,0);

\draw[
    very thick,
    blue!70!black,
    domain=0:1,
    samples=100,
    variable=\xx
]
    plot ({\xx},{6*(\xx)*(1-\xx)});

\draw[very thick,blue!70!black]
    (1,0)--(1.55,0);

\draw[very thick,dashed,red!75!black]
    (-1.55,0)--(-1,0);

\draw[
    very thick,
    dashed,
    red!75!black,
    domain=-1:1,
    samples=120,
    variable=\xx
]
    plot ({\xx},{0.75*(1-(\xx)*(\xx))});

\draw[very thick,dashed,red!75!black]
    (1,0)--(1.55,0);

\fill[red!75!black] (-1,0) circle (1.2pt);
\node[below left=2pt,font=\scriptsize]
    at (-1,0) {$(-1,0)$};

\fill[blue!70!black] (0,0) circle (1.2pt);
\node[below right=2pt,font=\scriptsize]
    at (0,0) {$(0,0)$};

\fill[red!75!black] (0,0.75) circle (1.2pt);
\node[left=3pt,font=\scriptsize]
    at (0,0.75) {$\left(0,\frac34\right)$};

\fill[blue!70!black] (0.5,1.5) circle (1.2pt);
\node[above right=2pt,font=\scriptsize]
    at (0.5,1.5) {$\left(\frac12,\frac32\right)$};

\fill[black] (1,0) circle (1.2pt);
\node[below right=2pt,font=\scriptsize]
    at (1,0) {$(1,0)$};

\draw[very thick,blue!70!black]
    (-1.45,1.62)--(-1.15,1.62);
\node[right=2pt,font=\small]
    at (-1.15,1.62) {$q'$};

\draw[very thick,dashed,red!75!black]
    (-0.65,1.62)--(-0.35,1.62);
\node[right=2pt,font=\small]
    at (-0.35,1.62) {$p'$};

\end{scope}

\node[font=\small] at (0,-0.9)
    {(b) The derivatives \(q'\) and \(p'\)};
\end{scope}

\end{tikzpicture}

\caption{The scalar functions \(q,p\) and their derivatives.}
\label{fig:ncpl-inner-gates}
\end{figure}
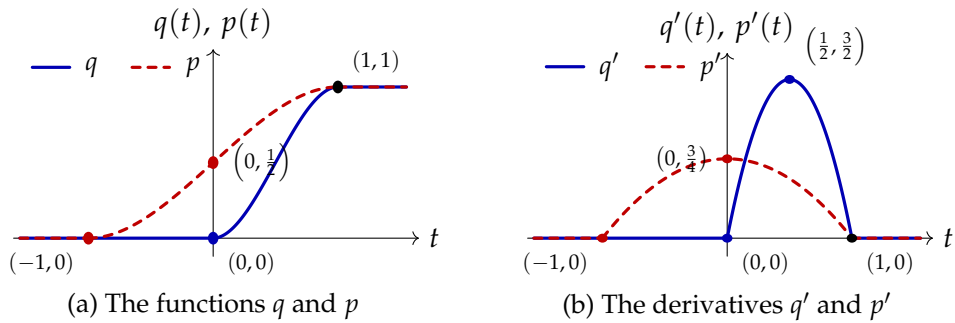

Let $\{\omega_j\}_{j=0}^N$ be positive weights, to be
specified below. For each $j\in[N]$, define the local dual interaction
\[
    d_j(u,v)
    :=
    \omega_{j-1}\bigl(1-q(u)p(v)\bigr)
    +\frac{\omega_j}{2}(v^-)^2.
\]
The dual chain is then defined by
\begin{equation*}
H(\y):=\sum_{j=1}^N d_j(y_{j-1},y_j).
\end{equation*}
The fixed coordinate $y_0=1$ initiates the chain. At a point
with $y_j=\cdots=y_N=0$, the identities $q(0)=q'(0)=0$ imply that
no gradient is generated in $y_{j+1},\ldots,y_N$. Once $y_{j-1}$ is
positive, however, $q(y_{j-1})>0$ and $p'(0)=3/4$, so the gradient can
reveal $y_j$. The quadratic term in $(y_j^-)^2$ supplies a gradient in
the negative region, where the derivatives of the gates may vanish.
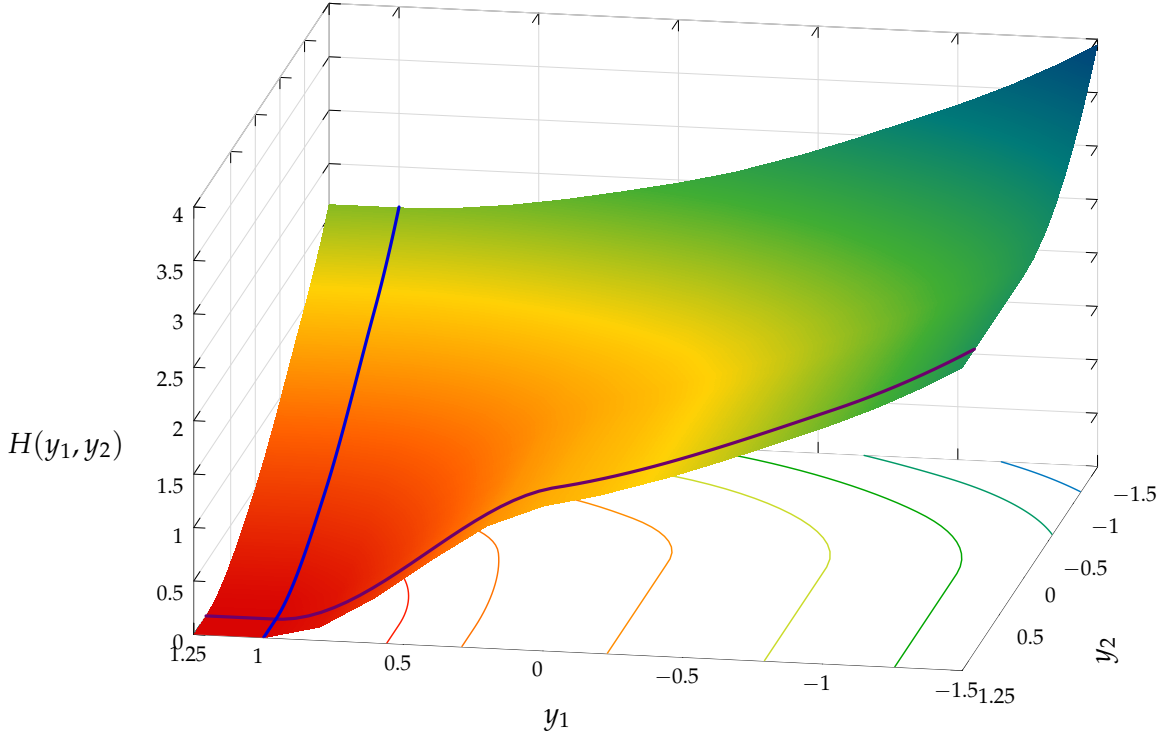
\begin{figure}[h]
\centering
\pgfplotsset{
    colormap={Hmap}{
        rgb255=(217,5,3)
        rgb255=(255,97,0)
        rgb255=(255,209,5)
        rgb255=(64,173,51)
        rgb255=(0,122,120)
        rgb255=(0,61,122)
    }
}

\begin{tikzpicture}[line cap=round,line join=round]
\begin{axis}[
    width=0.82\textwidth,
    height=0.66\textwidth,
    view={10}{24},
    xmin=-1.5,
    xmax=1.25,
    ymin=-1.5,
    ymax=1.25,
    zmin=0,
    zmax=4,
    x dir=reverse,
    y dir=reverse,
    xtick={-1.5,-1,-0.5,0,0.5,1,1.25},
    ytick={-1.5,-1,-0.5,0,0.5,1.25},
    ztick={0,0.5,...,4},
    xlabel={$y_1$},
    ylabel={$y_2$},
    zlabel={$H(y_1,y_2)$},
    axis lines=box,
    zlabel style={
        rotate=-90,
        anchor=east,
        xshift=-4pt,
        font=\normalsize
    },
    axis line style={
        black!65,
        line width=0.35pt
    },
    tick style={black},
    x tick label style={font=\scriptsize},
    y tick label style={font=\scriptsize,xshift=4pt},
    z tick label style={font=\scriptsize},
    label style={font=\normalsize},
    grid=major,
    grid style={
        black!15,
        line width=0.25pt
    },
    enlargelimits=false,
    scale mode=stretch to fill,
    z post scale=0.92,
    colormap name=Hmap,
    point meta min=0,
    point meta max=4
]

\path[fill=white,draw=none]
    (axis cs:-1.5,-1.5,0)
    -- (axis cs:1.25,-1.5,0)
    -- (axis cs:1.25,1.25,0)
    -- (axis cs:-1.5,1.25,0)
    -- cycle;

\addplot3[
    draw=red!78!orange,
    line width=0.55pt
] table[
    row sep=crcr,
    x=yone,
    y=ytwo,
    z=value
] {
    yone ytwo value\\
        1.2500000 0.0000012 0.0030000\\
        1.2028571 0.0000012 0.0030000\\
        1.1557143 0.0000012 0.0030000\\
        1.1085714 0.0000012 0.0030000\\
        1.0614286 0.0000012 0.0030000\\
        1.0142857 0.0000012 0.0030000\\
        0.9710714 0.0022330 0.0030000\\
        0.9317857 0.0122627 0.0030000\\
        0.9003571 0.0260561 0.0030000\\
        0.8706801 0.0439286 0.0030000\\
        0.8447639 0.0635714 0.0030000\\
        0.8189776 0.0871429 0.0030000\\
        0.7969631 0.1107143 0.0030000\\
        0.7746429 0.1381807 0.0030000\\
        0.7550000 0.1656667 0.0030000\\
        0.7374197 0.1932143 0.0030000\\
        0.7196429 0.2242332 0.0030000\\
        0.7039286 0.2546784 0.0030000\\
        0.6885998 0.2875000 0.0030000\\
        0.6752672 0.3189286 0.0030000\\
        0.6616300 0.3542857 0.0030000\\
        0.6489286 0.3906403 0.0030000\\
        0.6371429 0.4278985 0.0030000\\
        0.6267095 0.4642857 0.0030000\\
        0.6174899 0.4996429 0.0030000\\
        0.6082255 0.5389286 0.0030000\\
        0.5999054 0.5782143 0.0030000\\
        0.5924638 0.6175000 0.0030000\\
        0.5858454 0.6567857 0.0030000\\
        0.5794682 0.7000000 0.0030000\\
        0.5742857 0.7406841 0.0030000\\
        0.5697623 0.7825000 0.0030000\\
        0.5659208 0.8257143 0.0030000\\
        0.5626821 0.8728571 0.0030000\\
        0.5605863 0.9160714 0.0030000\\
        0.5592603 0.9632143 0.0030000\\
        0.5589400 1.0103571 0.0030000\\
        0.5589400 1.0575000 0.0030000\\
        0.5589400 1.1046429 0.0030000\\
        0.5589400 1.1517857 0.0030000\\
        0.5589400 1.1989286 0.0030000\\
        0.5589400 1.2500000 0.0030000\\
};

\addplot3[
    draw=orange!90!red,
    line width=0.55pt
] table[
    row sep=crcr,
    x=yone,
    y=ytwo,
    z=value
] {
    yone ytwo value\\
        1.2500000 -0.5293165 0.0030000\\
        1.1832143 -0.5293165 0.0030000\\
        1.1164286 -0.5293165 0.0030000\\
        1.0496429 -0.5293165 0.0030000\\
        0.9828571 -0.5290571 0.0030000\\
        0.9200000 -0.5237910 0.0030000\\
        0.8650000 -0.5137446 0.0030000\\
        0.8100000 -0.4985521 0.0030000\\
        0.7628571 -0.4811979 0.0030000\\
        0.7157143 -0.4594597 0.0030000\\
        0.6727763 -0.4353571 0.0030000\\
        0.6332143 -0.4089422 0.0030000\\
        0.5967391 -0.3803571 0.0030000\\
        0.5623682 -0.3489286 0.0030000\\
        0.5292994 -0.3135714 0.0030000\\
        0.4996429 -0.2764038 0.0030000\\
        0.4721429 -0.2358331 0.0030000\\
        0.4485714 -0.1946684 0.0030000\\
        0.4265444 -0.1485714 0.0030000\\
        0.4081055 -0.1014286 0.0030000\\
        0.3921532 -0.0503571 0.0030000\\
        0.3790331 0.0046429 0.0030000\\
        0.3681660 0.0596429 0.0030000\\
        0.3582143 0.1151586 0.0030000\\
        0.3487389 0.1735714 0.0030000\\
        0.3401225 0.2325000 0.0030000\\
        0.3323690 0.2914286 0.0030000\\
        0.3254082 0.3503571 0.0030000\\
        0.3189286 0.4118402 0.0030000\\
        0.3132914 0.4721429 0.0030000\\
        0.3081322 0.5350000 0.0030000\\
        0.3036642 0.5978571 0.0030000\\
        0.2998573 0.6607143 0.0030000\\
        0.2966914 0.7235714 0.0030000\\
        0.2941573 0.7864286 0.0030000\\
        0.2921546 0.8532143 0.0030000\\
        0.2909269 0.9160714 0.0030000\\
        0.2903418 0.9828571 0.0030000\\
        0.2903158 1.0496429 0.0030000\\
        0.2903158 1.1164286 0.0030000\\
        0.2903158 1.1832143 0.0030000\\
        0.2903158 1.2500000 0.0030000\\
};

\addplot3[
    draw=orange!90!yellow,
    line width=0.55pt
] table[
    row sep=crcr,
    x=yone,
    y=ytwo,
    z=value
] {
    yone ytwo value\\
        1.2500000 -1.0000001 0.0030000\\
        1.1596429 -1.0000001 0.0030000\\
        1.0692857 -1.0000001 0.0030000\\
        0.9789286 -0.9997665 0.0030000\\
        0.8925000 -0.9940994 0.0030000\\
        0.8139286 -0.9828305 0.0030000\\
        0.7392857 -0.9672111 0.0030000\\
        0.6646429 -0.9471464 0.0030000\\
        0.5978571 -0.9255751 0.0030000\\
        0.5310714 -0.9005868 0.0030000\\
        0.4682143 -0.8738231 0.0030000\\
        0.4058545 -0.8439286 0.0030000\\
        0.3472597 -0.8125000 0.0030000\\
        0.2914286 -0.7792007 0.0030000\\
        0.2349233 -0.7417857 0.0030000\\
        0.1814286 -0.7025617 0.0030000\\
        0.1324930 -0.6632143 0.0030000\\
        0.0832143 -0.6202086 0.0030000\\
        0.0366720 -0.5767857 0.0030000\\
        -0.0076691 -0.5335714 0.0030000\\
        -0.0506823 -0.4864286 0.0030000\\
        -0.0928015 -0.4314286 0.0030000\\
        -0.1278345 -0.3764286 0.0030000\\
        -0.1586885 -0.3175000 0.0030000\\
        -0.1848848 -0.2546429 0.0030000\\
        -0.2068545 -0.1839286 0.0030000\\
        -0.2220448 -0.1092857 0.0030000\\
        -0.2297425 -0.0267857 0.0030000\\
        -0.2302329 0.0675000 0.0030000\\
        -0.2302329 0.1578571 0.0030000\\
        -0.2302329 0.2482143 0.0030000\\
        -0.2302329 0.3385714 0.0030000\\
        -0.2302329 0.4289286 0.0030000\\
        -0.2302329 0.5192857 0.0030000\\
        -0.2302329 0.6096429 0.0030000\\
        -0.2302329 0.7039286 0.0030000\\
        -0.2302329 0.7942857 0.0030000\\
        -0.2302329 0.8846429 0.0030000\\
        -0.2302329 0.9750000 0.0030000\\
        -0.2302329 1.0653571 0.0030000\\
        -0.2302329 1.1557143 0.0030000\\
        -0.2302329 1.2500000 0.0030000\\
};

\addplot3[
    draw=yellow!75!green,
    line width=0.55pt
] table[
    row sep=crcr,
    x=yone,
    y=ytwo,
    z=value
] {
    yone ytwo value\\
        1.2500000 -1.4142128 0.0030000\\
        1.1360714 -1.4142128 0.0030000\\
        1.0221429 -1.4142128 0.0030000\\
        0.9121429 -1.4113996 0.0030000\\
        0.8100000 -1.4014757 0.0030000\\
        0.7078571 -1.3850226 0.0030000\\
        0.6135714 -1.3645558 0.0030000\\
        0.5232143 -1.3405976 0.0030000\\
        0.4367857 -1.3140430 0.0030000\\
        0.3503571 -1.2842495 0.0030000\\
        0.2663883 -1.2525000 0.0030000\\
        0.1853571 -1.2195241 0.0030000\\
        0.1067857 -1.1856331 0.0030000\\
        0.0282143 -1.1501349 0.0030000\\
        -0.0525173 -1.1110714 0.0030000\\
        -0.1250000 -1.0702208 0.0030000\\
        -0.1955441 -1.0246429 0.0030000\\
        -0.2601502 -0.9775000 0.0030000\\
        -0.3227575 -0.9264286 0.0030000\\
        -0.3842857 -0.8704495 0.0030000\\
        -0.4414755 -0.8125000 0.0030000\\
        -0.4942857 -0.7529940 0.0030000\\
        -0.5453571 -0.6887537 0.0030000\\
        -0.5938523 -0.6200000 0.0030000\\
        -0.6376051 -0.5492857 0.0030000\\
        -0.6774739 -0.4746429 0.0030000\\
        -0.7127697 -0.3960714 0.0030000\\
        -0.7427266 -0.3135714 0.0030000\\
        -0.7682789 -0.2192857 0.0030000\\
        -0.7850000 -0.1227718 0.0030000\\
        -0.7925072 -0.0150000 0.0030000\\
        -0.7926209 0.0989286 0.0030000\\
        -0.7926209 0.2167857 0.0030000\\
        -0.7926209 0.3307143 0.0030000\\
        -0.7926209 0.4446429 0.0030000\\
        -0.7926209 0.5585714 0.0030000\\
        -0.7926209 0.6725000 0.0030000\\
        -0.7926209 0.7903571 0.0030000\\
        -0.7926209 0.9042857 0.0030000\\
        -0.7926209 1.0182143 0.0030000\\
        -0.7926209 1.1321429 0.0030000\\
        -0.7926209 1.2500000 0.0030000\\
};

\addplot3[
    draw=green!65!black,
    line width=0.55pt
] table[
    row sep=crcr,
    x=yone,
    y=ytwo,
    z=value
] {
    yone ytwo value\\
        -0.0390228 -1.5000000 0.0030000\\
        -0.1071006 -1.4725000 0.0030000\\
        -0.1760714 -1.4412042 0.0030000\\
        -0.2389286 -1.4096831 0.0030000\\
        -0.3017857 -1.3752967 0.0030000\\
        -0.3631723 -1.3389286 0.0030000\\
        -0.4196429 -1.3030042 0.0030000\\
        -0.4785714 -1.2629369 0.0030000\\
        -0.5362792 -1.2210714 0.0030000\\
        -0.5885714 -1.1807989 0.0030000\\
        -0.6435714 -1.1359301 0.0030000\\
        -0.6946429 -1.0918283 0.0030000\\
        -0.7457143 -1.0452101 0.0030000\\
        -0.7955026 -0.9971429 0.0030000\\
        -0.8417018 -0.9500000 0.0030000\\
        -0.8889888 -0.8989286 0.0030000\\
        -0.9342857 -0.8469793 0.0030000\\
        -0.9775000 -0.7942352 0.0030000\\
        -1.0204667 -0.7378571 0.0030000\\
        -1.0600000 -0.6798409 0.0030000\\
        -1.0953571 -0.6212695 0.0030000\\
        -1.1293249 -0.5571429 0.0030000\\
        -1.1582176 -0.4942857 0.0030000\\
        -1.1859060 -0.4235714 0.0030000\\
        -1.2099723 -0.3489286 0.0030000\\
        -1.2290457 -0.2742857 0.0030000\\
        -1.2446429 -0.1914402 0.0030000\\
        -1.2551874 -0.1014286 0.0030000\\
        -1.2592302 -0.0110714 0.0030000\\
        -1.2592789 0.0871429 0.0030000\\
        -1.2592789 0.1853571 0.0030000\\
        -1.2592789 0.2796429 0.0030000\\
        -1.2592789 0.3778571 0.0030000\\
        -1.2592789 0.4721429 0.0030000\\
        -1.2592789 0.5703571 0.0030000\\
        -1.2592789 0.6685714 0.0030000\\
        -1.2592789 0.7628571 0.0030000\\
        -1.2592789 0.8610714 0.0030000\\
        -1.2592789 0.9592857 0.0030000\\
        -1.2592789 1.0535714 0.0030000\\
        -1.2592789 1.1517857 0.0030000\\
        -1.2592789 1.2500000 0.0030000\\
};

\addplot3[
    draw=teal!80!green,
    line width=0.55pt
] table[
    row sep=crcr,
    x=yone,
    y=ytwo,
    z=value
] {
    yone ytwo value\\
        -0.6646362 -1.5000000 0.0030000\\
        -0.6889553 -1.4842857 0.0030000\\
        -0.7142857 -1.4676232 0.0030000\\
        -0.7417857 -1.4491969 0.0030000\\
        -0.7653571 -1.4331239 0.0030000\\
        -0.7928571 -1.4140477 0.0030000\\
        -0.8164286 -1.3974208 0.0030000\\
        -0.8432201 -1.3782143 0.0030000\\
        -0.8675000 -1.3605238 0.0030000\\
        -0.8913918 -1.3428571 0.0030000\\
        -0.9175523 -1.3232143 0.0030000\\
        -0.9382143 -1.3074808 0.0030000\\
        -0.9636369 -1.2878571 0.0030000\\
        -0.9887131 -1.2682143 0.0030000\\
        -1.0128571 -1.2489618 0.0030000\\
        -1.0364286 -1.2294716 0.0030000\\
        -1.0599123 -1.2092857 0.0030000\\
        -1.0819115 -1.1896429 0.0030000\\
        -1.1032143 -1.1699164 0.0030000\\
        -1.1267857 -1.1472305 0.0030000\\
        -1.1472305 -1.1267857 0.0030000\\
        -1.1661992 -1.1071429 0.0030000\\
        -1.1881312 -1.0835714 0.0030000\\
        -1.2092089 -1.0600000 0.0030000\\
        -1.2289286 -1.0370723 0.0030000\\
        -1.2485714 -1.0133381 0.0030000\\
        -1.2677139 -0.9892857 0.0030000\\
        -1.2857599 -0.9657143 0.0030000\\
        -1.3035714 -0.9415338 0.0030000\\
        -1.3225781 -0.9146429 0.0030000\\
        -1.3389286 -0.8905363 0.0030000\\
        -1.3546429 -0.8664445 0.0030000\\
        -1.3711980 -0.8400000 0.0030000\\
        -1.3876694 -0.8125000 0.0030000\\
        -1.4034093 -0.7850000 0.0030000\\
        -1.4184419 -0.7575000 0.0030000\\
        -1.4332143 -0.7291644 0.0030000\\
        -1.4483724 -0.6985714 0.0030000\\
        -1.4613171 -0.6710714 0.0030000\\
        -1.4753442 -0.6396429 0.0030000\\
        -1.4882143 -0.6090993 0.0030000\\
        -1.5000000 -0.5794679 0.0030000\\
};

\addplot3[
    draw=cyan!55!blue,
    line width=0.55pt
] table[
    row sep=crcr,
    x=yone,
    y=ytwo,
    z=value
] {
    yone ytwo value\\
        -1.1557607 -1.5000000 0.0030000\\
        -1.1658883 -1.4921429 0.0030000\\
        -1.1739286 -1.4858245 0.0030000\\
        -1.1817857 -1.4795832 0.0030000\\
        -1.1935714 -1.4700918 0.0030000\\
        -1.2014286 -1.4636779 0.0030000\\
        -1.2097769 -1.4567857 0.0030000\\
        -1.2191754 -1.4489286 0.0030000\\
        -1.2289286 -1.4406664 0.0030000\\
        -1.2376109 -1.4332143 0.0030000\\
        -1.2466514 -1.4253571 0.0030000\\
        -1.2555786 -1.4175000 0.0030000\\
        -1.2643943 -1.4096429 0.0030000\\
        -1.2730987 -1.4017857 0.0030000\\
        -1.2816964 -1.3939286 0.0030000\\
        -1.2917857 -1.3845839 0.0030000\\
        -1.2996429 -1.3772117 0.0030000\\
        -1.3075000 -1.3697548 0.0030000\\
        -1.3153571 -1.3622118 0.0030000\\
        -1.3271429 -1.3507325 0.0030000\\
        -1.3350000 -1.3429672 0.0030000\\
        -1.3428571 -1.3351107 0.0030000\\
        -1.3507143 -1.3271614 0.0030000\\
        -1.3585714 -1.3191169 0.0030000\\
        -1.3664286 -1.3109759 0.0030000\\
        -1.3742857 -1.3027366 0.0030000\\
        -1.3821429 -1.2943972 0.0030000\\
        -1.3918728 -1.2839286 0.0030000\\
        -1.3990800 -1.2760714 0.0030000\\
        -1.4062067 -1.2682143 0.0030000\\
        -1.4167467 -1.2564286 0.0030000\\
        -1.4236755 -1.2485714 0.0030000\\
        -1.4305283 -1.2407143 0.0030000\\
        -1.4373062 -1.2328571 0.0030000\\
        -1.4473311 -1.2210714 0.0030000\\
        -1.4539238 -1.2132143 0.0030000\\
        -1.4607143 -1.2050307 0.0030000\\
        -1.4685714 -1.1954416 0.0030000\\
        -1.4764376 -1.1857143 0.0030000\\
        -1.4842857 -1.1758735 0.0030000\\
        -1.4919998 -1.1660714 0.0030000\\
        -1.5000000 -1.1557607 0.0030000\\
};

\addplot3[
    surf,
    shader=interp,
    draw=none,
    mesh/rows=17,
    mesh/cols=17,
    mesh/ordering=x varies,
    point meta=explicit
] table[
    row sep=crcr,
    x=yone,
    y=ytwo,
    z=value,
    meta=meta
] {
    yone ytwo value meta\\
        -1.5000000 -1.5000000 3.9571068 3.9571068\\
        -1.3500000 -1.5000000 3.7433568 3.7433568\\
        -1.2000000 -1.5000000 3.5521068 3.5521068\\
        -1.0500000 -1.5000000 3.3833568 3.3833568\\
        -1.0000000 -1.5000000 3.3321068 3.3321068\\
        -0.8500000 -1.5000000 3.1820210 3.1820210\\
        -0.6500000 -1.5000000 2.9859706 2.9859706\\
        -0.4500000 -1.5000000 2.8023432 2.8023432\\
        -0.2000000 -1.5000000 2.6032052 2.6032052\\
        0.0000000 -1.5000000 2.4785534 2.4785534\\
        0.2000000 -1.5000000 2.3739016 2.3739016\\
        0.4000000 -1.5000000 2.2777351 2.2777351\\
        0.6000000 -1.5000000 2.1985391 2.1985391\\
        0.8000000 -1.5000000 2.1447990 2.1447990\\
        1.0000000 -1.5000000 2.1250000 2.1250000\\
        1.1250000 -1.5000000 2.1250000 2.1250000\\
        1.2500000 -1.5000000 2.1250000 2.1250000\\
        -1.5000000 -1.3500000 3.7433568 3.7433568\\
        -1.3500000 -1.3500000 3.5296068 3.5296068\\
        -1.2000000 -1.3500000 3.3383568 3.3383568\\
        -1.0500000 -1.3500000 3.1696068 3.1696068\\
        -1.0000000 -1.3500000 3.1183568 3.1183568\\
        -0.8500000 -1.3500000 2.9682710 2.9682710\\
        -0.6500000 -1.3500000 2.7722206 2.7722206\\
        -0.4500000 -1.3500000 2.5885932 2.5885932\\
        -0.2000000 -1.3500000 2.3894552 2.3894552\\
        0.0000000 -1.3500000 2.2648034 2.2648034\\
        0.2000000 -1.3500000 2.1601516 2.1601516\\
        0.4000000 -1.3500000 2.0639851 2.0639851\\
        0.6000000 -1.3500000 1.9847891 1.9847891\\
        0.8000000 -1.3500000 1.9310490 1.9310490\\
        1.0000000 -1.3500000 1.9112500 1.9112500\\
        1.1250000 -1.3500000 1.9112500 1.9112500\\
        1.2500000 -1.3500000 1.9112500 1.9112500\\
        -1.5000000 -1.2000000 3.5521068 3.5521068\\
        -1.3500000 -1.2000000 3.3383568 3.3383568\\
        -1.2000000 -1.2000000 3.1471068 3.1471068\\
        -1.0500000 -1.2000000 2.9783568 2.9783568\\
        -1.0000000 -1.2000000 2.9271068 2.9271068\\
        -0.8500000 -1.2000000 2.7770210 2.7770210\\
        -0.6500000 -1.2000000 2.5809706 2.5809706\\
        -0.4500000 -1.2000000 2.3973432 2.3973432\\
        -0.2000000 -1.2000000 2.1982052 2.1982052\\
        0.0000000 -1.2000000 2.0735534 2.0735534\\
        0.2000000 -1.2000000 1.9689016 1.9689016\\
        0.4000000 -1.2000000 1.8727351 1.8727351\\
        0.6000000 -1.2000000 1.7935391 1.7935391\\
        0.8000000 -1.2000000 1.7397990 1.7397990\\
        1.0000000 -1.2000000 1.7200000 1.7200000\\
        1.1250000 -1.2000000 1.7200000 1.7200000\\
        1.2500000 -1.2000000 1.7200000 1.7200000\\
        -1.5000000 -1.0500000 3.3833568 3.3833568\\
        -1.3500000 -1.0500000 3.1696068 3.1696068\\
        -1.2000000 -1.0500000 2.9783568 2.9783568\\
        -1.0500000 -1.0500000 2.8096068 2.8096068\\
        -1.0000000 -1.0500000 2.7583568 2.7583568\\
        -0.8500000 -1.0500000 2.6082710 2.6082710\\
        -0.6500000 -1.0500000 2.4122206 2.4122206\\
        -0.4500000 -1.0500000 2.2285932 2.2285932\\
        -0.2000000 -1.0500000 2.0294552 2.0294552\\
        0.0000000 -1.0500000 1.9048034 1.9048034\\
        0.2000000 -1.0500000 1.8001516 1.8001516\\
        0.4000000 -1.0500000 1.7039851 1.7039851\\
        0.6000000 -1.0500000 1.6247891 1.6247891\\
        0.8000000 -1.0500000 1.5710490 1.5710490\\
        1.0000000 -1.0500000 1.5512500 1.5512500\\
        1.1250000 -1.0500000 1.5512500 1.5512500\\
        1.2500000 -1.0500000 1.5512500 1.5512500\\
        -1.5000000 -1.0000000 3.3321068 3.3321068\\
        -1.3500000 -1.0000000 3.1183568 3.1183568\\
        -1.2000000 -1.0000000 2.9271068 2.9271068\\
        -1.0500000 -1.0000000 2.7583568 2.7583568\\
        -1.0000000 -1.0000000 2.7071068 2.7071068\\
        -0.8500000 -1.0000000 2.5570210 2.5570210\\
        -0.6500000 -1.0000000 2.3609706 2.3609706\\
        -0.4500000 -1.0000000 2.1773432 2.1773432\\
        -0.2000000 -1.0000000 1.9782052 1.9782052\\
        0.0000000 -1.0000000 1.8535534 1.8535534\\
        0.2000000 -1.0000000 1.7489016 1.7489016\\
        0.4000000 -1.0000000 1.6527351 1.6527351\\
        0.6000000 -1.0000000 1.5735391 1.5735391\\
        0.8000000 -1.0000000 1.5197990 1.5197990\\
        1.0000000 -1.0000000 1.5000000 1.5000000\\
        1.1250000 -1.0000000 1.5000000 1.5000000\\
        1.2500000 -1.0000000 1.5000000 1.5000000\\
        -1.5000000 -0.8500000 3.1933568 3.1933568\\
        -1.3500000 -0.8500000 2.9796068 2.9796068\\
        -1.2000000 -0.8500000 2.7883568 2.7883568\\
        -1.0500000 -0.8500000 2.6196068 2.6196068\\
        -1.0000000 -0.8500000 2.5683568 2.5683568\\
        -0.8500000 -0.8500000 2.4182710 2.4182710\\
        -0.6500000 -0.8500000 2.2222206 2.2222206\\
        -0.4500000 -0.8500000 2.0385932 2.0385932\\
        -0.2000000 -0.8500000 1.8394552 1.8394552\\
        0.0000000 -0.8500000 1.7148034 1.7148034\\
        0.2000000 -0.8500000 1.6084843 1.6084843\\
        0.4000000 -0.8500000 1.5083421 1.5083421\\
        0.6000000 -0.8500000 1.4244009 1.4244009\\
        0.8000000 -0.8500000 1.3666850 1.3666850\\
        1.0000000 -0.8500000 1.3452187 1.3452187\\
        1.1250000 -0.8500000 1.3452187 1.3452187\\
        1.2500000 -0.8500000 1.3452187 1.3452187\\
        -1.5000000 -0.6500000 3.0433568 3.0433568\\
        -1.3500000 -0.6500000 2.8296068 2.8296068\\
        -1.2000000 -0.6500000 2.6383568 2.6383568\\
        -1.0500000 -0.6500000 2.4696068 2.4696068\\
        -1.0000000 -0.6500000 2.4183568 2.4183568\\
        -0.8500000 -0.6500000 2.2682710 2.2682710\\
        -0.6500000 -0.6500000 2.0722206 2.0722206\\
        -0.4500000 -0.6500000 1.8885932 1.8885932\\
        -0.2000000 -0.6500000 1.6894552 1.6894552\\
        0.0000000 -0.6500000 1.5648034 1.5648034\\
        0.2000000 -0.6500000 1.4517113 1.4517113\\
        0.4000000 -0.6500000 1.3354181 1.3354181\\
        0.6000000 -0.6500000 1.2321999 1.2321999\\
        0.8000000 -0.6500000 1.1583330 1.1583330\\
        1.0000000 -0.6500000 1.1300937 1.1300937\\
        1.1250000 -0.6500000 1.1300937 1.1300937\\
        1.2500000 -0.6500000 1.1300937 1.1300937\\
        -1.5000000 -0.4500000 2.9333568 2.9333568\\
        -1.3500000 -0.4500000 2.7196068 2.7196068\\
        -1.2000000 -0.4500000 2.5283568 2.5283568\\
        -1.0500000 -0.4500000 2.3596068 2.3596068\\
        -1.0000000 -0.4500000 2.3083568 2.3083568\\
        -0.8500000 -0.4500000 2.1582710 2.1582710\\
        -0.6500000 -0.4500000 1.9622206 1.9622206\\
        -0.4500000 -0.4500000 1.7785932 1.7785932\\
        -0.2000000 -0.4500000 1.5794552 1.5794552\\
        0.0000000 -0.4500000 1.4548034 1.4548034\\
        0.2000000 -0.4500000 1.3308823 1.3308823\\
        0.4000000 -0.4500000 1.1887661 1.1887661\\
        0.6000000 -0.4500000 1.0547269 1.0547269\\
        0.8000000 -0.4500000 0.9550370 0.9550370\\
        1.0000000 -0.4500000 0.9159687 0.9159687\\
        1.1250000 -0.4500000 0.9159687 0.9159687\\
        1.2500000 -0.4500000 0.9159687 0.9159687\\
        -1.5000000 -0.2000000 2.8521068 2.8521068\\
        -1.3500000 -0.2000000 2.6383568 2.6383568\\
        -1.2000000 -0.2000000 2.4471068 2.4471068\\
        -1.0500000 -0.2000000 2.2783568 2.2783568\\
        -1.0000000 -0.2000000 2.2271068 2.2271068\\
        -0.8500000 -0.2000000 2.0770210 2.0770210\\
        -0.6500000 -0.2000000 1.8809706 1.8809706\\
        -0.4500000 -0.2000000 1.6973432 1.6973432\\
        -0.2000000 -0.2000000 1.4982052 1.4982052\\
        0.0000000 -0.2000000 1.3735534 1.3735534\\
        0.2000000 -0.2000000 1.2322936 1.2322936\\
        0.4000000 -0.2000000 1.0488311 1.0488311\\
        0.6000000 -0.2000000 0.8654431 0.8654431\\
        0.8000000 -0.2000000 0.7244070 0.7244070\\
        1.0000000 -0.2000000 0.6680000 0.6680000\\
        1.1250000 -0.2000000 0.6680000 0.6680000\\
        1.2500000 -0.2000000 0.6680000 0.6680000\\
        -1.5000000 0.0000000 2.8321068 2.8321068\\
        -1.3500000 0.0000000 2.6183568 2.6183568\\
        -1.2000000 0.0000000 2.4271068 2.4271068\\
        -1.0500000 0.0000000 2.2583568 2.2583568\\
        -1.0000000 0.0000000 2.2071068 2.2071068\\
        -0.8500000 0.0000000 2.0570210 2.0570210\\
        -0.6500000 0.0000000 1.8609706 1.8609706\\
        -0.4500000 0.0000000 1.6773432 1.6773432\\
        -0.2000000 0.0000000 1.4782052 1.4782052\\
        0.0000000 0.0000000 1.3535534 1.3535534\\
        0.2000000 0.0000000 1.1969016 1.1969016\\
        0.4000000 0.0000000 0.9767351 0.9767351\\
        0.6000000 0.0000000 0.7495391 0.7495391\\
        0.8000000 0.0000000 0.5717990 0.5717990\\
        1.0000000 0.0000000 0.5000000 0.5000000\\
        1.1250000 0.0000000 0.5000000 0.5000000\\
        1.2500000 0.0000000 0.5000000 0.5000000\\
        -1.5000000 0.2000000 2.8321068 2.8321068\\
        -1.3500000 0.2000000 2.6183568 2.6183568\\
        -1.2000000 0.2000000 2.4271068 2.4271068\\
        -1.0500000 0.2000000 2.2583568 2.2583568\\
        -1.0000000 0.2000000 2.2071068 2.2071068\\
        -0.8500000 0.2000000 2.0570210 2.0570210\\
        -0.6500000 0.2000000 1.8609706 1.8609706\\
        -0.4500000 0.2000000 1.6773432 1.6773432\\
        -0.2000000 0.2000000 1.4782052 1.4782052\\
        0.0000000 0.2000000 1.3535534 1.3535534\\
        0.2000000 0.2000000 1.1815096 1.1815096\\
        0.4000000 0.2000000 0.9246391 0.9246391\\
        0.6000000 0.2000000 0.6536351 0.6536351\\
        0.8000000 0.2000000 0.4391910 0.4391910\\
        1.0000000 0.2000000 0.3520000 0.3520000\\
        1.1250000 0.2000000 0.3520000 0.3520000\\
        1.2500000 0.2000000 0.3520000 0.3520000\\
        -1.5000000 0.4000000 2.8321068 2.8321068\\
        -1.3500000 0.4000000 2.6183568 2.6183568\\
        -1.2000000 0.4000000 2.4271068 2.4271068\\
        -1.0500000 0.4000000 2.2583568 2.2583568\\
        -1.0000000 0.4000000 2.2071068 2.2071068\\
        -0.8500000 0.4000000 2.0570210 2.0570210\\
        -0.6500000 0.4000000 1.8609706 1.8609706\\
        -0.4500000 0.4000000 1.6773432 1.6773432\\
        -0.2000000 0.4000000 1.4782052 1.4782052\\
        0.0000000 0.4000000 1.3535534 1.3535534\\
        0.2000000 0.4000000 1.1673656 1.1673656\\
        0.4000000 0.4000000 0.8767671 0.8767671\\
        0.6000000 0.4000000 0.5655071 0.5655071\\
        0.8000000 0.4000000 0.3173350 0.3173350\\
        1.0000000 0.4000000 0.2160000 0.2160000\\
        1.1250000 0.4000000 0.2160000 0.2160000\\
        1.2500000 0.4000000 0.2160000 0.2160000\\
        -1.5000000 0.6000000 2.8321068 2.8321068\\
        -1.3500000 0.6000000 2.6183568 2.6183568\\
        -1.2000000 0.6000000 2.4271068 2.4271068\\
        -1.0500000 0.6000000 2.2583568 2.2583568\\
        -1.0000000 0.6000000 2.2071068 2.2071068\\
        -0.8500000 0.6000000 2.0570210 2.0570210\\
        -0.6500000 0.6000000 1.8609706 1.8609706\\
        -0.4500000 0.6000000 1.6773432 1.6773432\\
        -0.2000000 0.6000000 1.4782052 1.4782052\\
        0.0000000 0.6000000 1.3535534 1.3535534\\
        0.2000000 0.6000000 1.1557176 1.1557176\\
        0.4000000 0.6000000 0.8373431 0.8373431\\
        0.6000000 0.6000000 0.4929311 0.4929311\\
        0.8000000 0.6000000 0.2169830 0.2169830\\
        1.0000000 0.6000000 0.1040000 0.1040000\\
        1.1250000 0.6000000 0.1040000 0.1040000\\
        1.2500000 0.6000000 0.1040000 0.1040000\\
        -1.5000000 0.8000000 2.8321068 2.8321068\\
        -1.3500000 0.8000000 2.6183568 2.6183568\\
        -1.2000000 0.8000000 2.4271068 2.4271068\\
        -1.0500000 0.8000000 2.2583568 2.2583568\\
        -1.0000000 0.8000000 2.2071068 2.2071068\\
        -0.8500000 0.8000000 2.0570210 2.0570210\\
        -0.6500000 0.8000000 1.8609706 1.8609706\\
        -0.4500000 0.8000000 1.6773432 1.6773432\\
        -0.2000000 0.8000000 1.4782052 1.4782052\\
        0.0000000 0.8000000 1.3535534 1.3535534\\
        0.2000000 0.8000000 1.1478136 1.1478136\\
        0.4000000 0.8000000 0.8105911 0.8105911\\
        0.6000000 0.8000000 0.4436831 0.4436831\\
        0.8000000 0.8000000 0.1488870 0.1488870\\
        1.0000000 0.8000000 0.0280000 0.0280000\\
        1.1250000 0.8000000 0.0280000 0.0280000\\
        1.2500000 0.8000000 0.0280000 0.0280000\\
        -1.5000000 1.0000000 2.8321068 2.8321068\\
        -1.3500000 1.0000000 2.6183568 2.6183568\\
        -1.2000000 1.0000000 2.4271068 2.4271068\\
        -1.0500000 1.0000000 2.2583568 2.2583568\\
        -1.0000000 1.0000000 2.2071068 2.2071068\\
        -0.8500000 1.0000000 2.0570210 2.0570210\\
        -0.6500000 1.0000000 1.8609706 1.8609706\\
        -0.4500000 1.0000000 1.6773432 1.6773432\\
        -0.2000000 1.0000000 1.4782052 1.4782052\\
        0.0000000 1.0000000 1.3535534 1.3535534\\
        0.2000000 1.0000000 1.1449016 1.1449016\\
        0.4000000 1.0000000 0.8007351 0.8007351\\
        0.6000000 1.0000000 0.4255391 0.4255391\\
        0.8000000 1.0000000 0.1237990 0.1237990\\
        1.0000000 1.0000000 0.0000000 0.0000000\\
        1.1250000 1.0000000 0.0000000 0.0000000\\
        1.2500000 1.0000000 0.0000000 0.0000000\\
        -1.5000000 1.1250000 2.8321068 2.8321068\\
        -1.3500000 1.1250000 2.6183568 2.6183568\\
        -1.2000000 1.1250000 2.4271068 2.4271068\\
        -1.0500000 1.1250000 2.2583568 2.2583568\\
        -1.0000000 1.1250000 2.2071068 2.2071068\\
        -0.8500000 1.1250000 2.0570210 2.0570210\\
        -0.6500000 1.1250000 1.8609706 1.8609706\\
        -0.4500000 1.1250000 1.6773432 1.6773432\\
        -0.2000000 1.1250000 1.4782052 1.4782052\\
        0.0000000 1.1250000 1.3535534 1.3535534\\
        0.2000000 1.1250000 1.1449016 1.1449016\\
        0.4000000 1.1250000 0.8007351 0.8007351\\
        0.6000000 1.1250000 0.4255391 0.4255391\\
        0.8000000 1.1250000 0.1237990 0.1237990\\
        1.0000000 1.1250000 0.0000000 0.0000000\\
        1.1250000 1.1250000 0.0000000 0.0000000\\
        1.2500000 1.1250000 0.0000000 0.0000000\\
        -1.5000000 1.2500000 2.8321068 2.8321068\\
        -1.3500000 1.2500000 2.6183568 2.6183568\\
        -1.2000000 1.2500000 2.4271068 2.4271068\\
        -1.0500000 1.2500000 2.2583568 2.2583568\\
        -1.0000000 1.2500000 2.2071068 2.2071068\\
        -0.8500000 1.2500000 2.0570210 2.0570210\\
        -0.6500000 1.2500000 1.8609706 1.8609706\\
        -0.4500000 1.2500000 1.6773432 1.6773432\\
        -0.2000000 1.2500000 1.4782052 1.4782052\\
        0.0000000 1.2500000 1.3535534 1.3535534\\
        0.2000000 1.2500000 1.1449016 1.1449016\\
        0.4000000 1.2500000 0.8007351 0.8007351\\
        0.6000000 1.2500000 0.4255391 0.4255391\\
        0.8000000 1.2500000 0.1237990 0.1237990\\
        1.0000000 1.2500000 0.0000000 0.0000000\\
        1.1250000 1.2500000 0.0000000 0.0000000\\
        1.2500000 1.2500000 0.0000000 0.0000000\\
};

\addplot3[
    draw=violet!85!black,
    very thick
] table[
    row sep=crcr,
    x=yone,
    y=ytwo,
    z=value
] {
    yone ytwo value\\
        -1.5000000 1.0000000 2.8321068\\
        -1.4541667 1.0000000 2.7644071\\
        -1.4083333 1.0000000 2.6988082\\
        -1.3625000 1.0000000 2.6353099\\
        -1.3166667 1.0000000 2.5739123\\
        -1.2708333 1.0000000 2.5146155\\
        -1.2250000 1.0000000 2.4574193\\
        -1.1791667 1.0000000 2.4023238\\
        -1.1333333 1.0000000 2.3493290\\
        -1.0875000 1.0000000 2.2984349\\
        -1.0416667 1.0000000 2.2496415\\
        -1.0000000 1.0000000 2.2071068\\
        -0.9958333 1.0000000 2.2029396\\
        -0.9500000 1.0000000 2.1570531\\
        -0.9041667 1.0000000 2.1111505\\
        -0.8583333 1.0000000 2.0653340\\
        -0.8125000 1.0000000 2.0197058\\
        -0.7666667 1.0000000 1.9743679\\
        -0.7208333 1.0000000 1.9294224\\
        -0.6750000 1.0000000 1.8849716\\
        -0.6291667 1.0000000 1.8411174\\
        -0.5833333 1.0000000 1.7979621\\
        -0.5375000 1.0000000 1.7556078\\
        -0.4916667 1.0000000 1.7141565\\
        -0.4458333 1.0000000 1.6737105\\
        -0.4000000 1.0000000 1.6343717\\
        -0.3541667 1.0000000 1.5962424\\
        -0.3083333 1.0000000 1.5594247\\
        -0.2625000 1.0000000 1.5240206\\
        -0.2166667 1.0000000 1.4901324\\
        -0.1708333 1.0000000 1.4578621\\
        -0.1250000 1.0000000 1.4273119\\
        -0.0791667 1.0000000 1.3985838\\
        -0.0333333 1.0000000 1.3717801\\
        0.0000000 1.0000000 1.3535534\\
        0.0125000 1.0000000 1.3464598\\
        0.0583333 1.0000000 1.3128412\\
        0.1041667 1.0000000 1.2682190\\
        0.1500000 1.0000000 1.2138505\\
        0.1958333 1.0000000 1.1509933\\
        0.2416667 1.0000000 1.0809050\\
        0.2875000 1.0000000 1.0048429\\
        0.3333333 1.0000000 0.9240647\\
        0.3791667 1.0000000 0.8398278\\
        0.4250000 1.0000000 0.7533897\\
        0.4708333 1.0000000 0.6660080\\
        0.5166667 1.0000000 0.5789400\\
        0.5625000 1.0000000 0.4934435\\
        0.6083333 1.0000000 0.4107757\\
        0.6541667 1.0000000 0.3321943\\
        0.7000000 1.0000000 0.2589567\\
        0.7458333 1.0000000 0.1923205\\
        0.7916667 1.0000000 0.1335432\\
        0.8375000 1.0000000 0.0838822\\
        0.8833333 1.0000000 0.0445951\\
        0.9291667 1.0000000 0.0169393\\
        0.9750000 1.0000000 0.0021724\\
        1.0000000 1.0000000 0.0000000\\
        1.0208333 1.0000000 0.0000000\\
        1.0666667 1.0000000 0.0000000\\
        1.1125000 1.0000000 0.0000000\\
        1.1583333 1.0000000 0.0000000\\
        1.2041667 1.0000000 0.0000000\\
        1.2500000 1.0000000 0.0000000\\
};

\addplot3[
    blue!85!black,
    very thick
] table[
    row sep=crcr,
    x=yone,
    y=ytwo,
    z=value
] {
    yone ytwo value\\
        1.0000000 -1.5000000 2.1250000\\
        1.0000000 -1.4541667 2.0573003\\
        1.0000000 -1.4083333 1.9917014\\
        1.0000000 -1.3625000 1.9282031\\
        1.0000000 -1.3166667 1.8668056\\
        1.0000000 -1.2708333 1.8075087\\
        1.0000000 -1.2250000 1.7503125\\
        1.0000000 -1.1791667 1.6952170\\
        1.0000000 -1.1333333 1.6422222\\
        1.0000000 -1.0875000 1.5913281\\
        1.0000000 -1.0416667 1.5425347\\
        1.0000000 -1.0000000 1.5000000\\
        1.0000000 -0.9958333 1.4958290\\
        1.0000000 -0.9500000 1.4494063\\
        1.0000000 -0.9041667 1.4020907\\
        1.0000000 -0.8583333 1.3540268\\
        1.0000000 -0.8125000 1.3053589\\
        1.0000000 -0.7666667 1.2562315\\
        1.0000000 -0.7208333 1.2067890\\
        1.0000000 -0.6750000 1.1571758\\
        1.0000000 -0.6291667 1.1075363\\
        1.0000000 -0.5833333 1.0580150\\
        1.0000000 -0.5375000 1.0087563\\
        1.0000000 -0.4916667 0.9599047\\
        1.0000000 -0.4458333 0.9116044\\
        1.0000000 -0.4000000 0.8640000\\
        1.0000000 -0.3541667 0.8172359\\
        1.0000000 -0.3083333 0.7714565\\
        1.0000000 -0.2625000 0.7268062\\
        1.0000000 -0.2166667 0.6834294\\
        1.0000000 -0.1708333 0.6414706\\
        1.0000000 -0.1250000 0.6010742\\
        1.0000000 -0.0791667 0.5623846\\
        1.0000000 -0.0333333 0.5255463\\
        1.0000000 0.0000000 0.5000000\\
        1.0000000 0.0125000 0.4906255\\
        1.0000000 0.0583333 0.4562996\\
        1.0000000 0.1041667 0.4221576\\
        1.0000000 0.1500000 0.3883438\\
        1.0000000 0.1958333 0.3550026\\
        1.0000000 0.2416667 0.3222785\\
        1.0000000 0.2875000 0.2903159\\
        1.0000000 0.3333333 0.2592593\\
        1.0000000 0.3791667 0.2292529\\
        1.0000000 0.4250000 0.2004414\\
        1.0000000 0.4708333 0.1729691\\
        1.0000000 0.5166667 0.1469803\\
        1.0000000 0.5625000 0.1226196\\
        1.0000000 0.6083333 0.1000314\\
        1.0000000 0.6541667 0.0793600\\
        1.0000000 0.7000000 0.0607500\\
        1.0000000 0.7458333 0.0443457\\
        1.0000000 0.7916667 0.0302915\\
        1.0000000 0.8375000 0.0187319\\
        1.0000000 0.8833333 0.0098113\\
        1.0000000 0.9291667 0.0036742\\
        1.0000000 0.9750000 0.0004648\\
        1.0000000 1.0000000 0.0000000\\
        1.0000000 1.0208333 0.0000000\\
        1.0000000 1.0666667 0.0000000\\
        1.0000000 1.1125000 0.0000000\\
        1.0000000 1.1583333 0.0000000\\
        1.0000000 1.2041667 0.0000000\\
        1.0000000 1.2500000 0.0000000\\
};

\end{axis}
\end{tikzpicture}

\caption{The dual chain function $H(y_1,y_2)$ for $N=2$.}
\label{fig:ncpl-dual-chain-H}
\end{figure}

The weights are defined recursively by
\begin{equation}
\label{eq:weight-sequence}
\omega_{N-1}=\omega_N=1,
\quad \text{and} \quad 
\omega_{j-1}=\sqrt{\frac{\omega_j}{N}},
\quad
j=N-1,\ldots,1.
\end{equation}
The backward recursion controls the cumulative weights while
keeping the terminal coordinate at unit scale. We write
$\bm{\omega}:=(\omega_1,\ldots,\omega_N)$. The estimates needed below
are collected in Lemma~\ref{lem:ncpl-weight-estimates}.

The following two lemmas record the properties of $q$, $p$, and the weight sequence used in our analysis. Their proofs are deferred to Appendices~\ref{app:proof-ncpl-gate-bounds} and~\ref{app:proof-ncpl-weight-estimates}, respectively.
\begin{lemma}
\label{lem:ncpl-gate-bounds}
For every $t\in\R$,
\[
0\leq q(t)\leq1,
\qquad
0\leq q'(t)\leq\frac32,
\qquad\text{and}\qquad
\Lip(q')\leq6,
\]
and
\[
0\leq p(t)\leq1,
\qquad
0\leq p'(t)\leq\frac34,
\qquad\text{and}\qquad
\Lip(p')\leq\frac32.
\]
Moreover,
\[
p(t)\geq q(t),
\quad 
\text{and} \quad
1-q(t)
\leq
2\left((p'(t))^2+(t^-)^2\right).
\]
\end{lemma}
\begin{lemma}
\label{lem:ncpl-weight-estimates}
Let $N\geq2$, and let $\{\omega_j\}_{j=0}^N$ be the weight sequence defined by \eqref{eq:weight-sequence}. Then
\begin{equation}
\label{eq:monotone-weight}
\frac1N\leq\omega_0\leq\omega_1\leq\cdots
\leq\omega_{N-1}=\omega_N=1.
\end{equation}
Furthermore,
\begin{equation}
\sum_{j=1}^N\omega_{j-1}<3,
\quad \text{and} \quad
\sum_{j=1}^N\frac{\omega_{j-1}^2}{\omega_j}<2.
\label{eq:ncpl-weight-sums}
\end{equation}
\end{lemma}

The preceding bounds yield the first key estimate for the dual
chain: a strengthened weighted P\L{} inequality that controls both
$H(\y)$ and the terminal activation deficit $1-q(y_N)$.
\begin{proposition}[Strengthened weighted P\L{} estimate for the dual chain]
\label{lem:ncpl-inner-technical}
For every $N\geq2$ and $\y\in\R^N$,
\begin{equation*}
H(\y)+10\bigl(1-q(y_N)\bigr)
\leq
320N\norm{\nabla H(\y)}_{*,\bm{\omega}}^2.
\end{equation*}
Moreover, $\nabla_{y_j}H(\y)\leq0$ for all $j\in[N]$.
\end{proposition}

\begin{proof}[Proof of Proposition~\ref{lem:ncpl-inner-technical}]
Direct differentiation of \(H\) yields 
\begin{equation}
    \label{eq:gradient-H}
    \nabla_{y_j}H(\y)
=
\begin{cases}
-\omega_{j-1}q(y_{j-1})p'(y_j)
-\omega_j y_j^-
-\omega_jq'(y_j)p(y_{j+1}),
& j\in[N-1],\\[2mm]
-\omega_{N-1}q(y_{N-1})p'(y_N)
-\omega_N y_N^-,
& j=N.
\end{cases}
\end{equation}
Since $q,p,q'$, and $p'$ are nonnegative,
$\nabla_{y_j}H(\y)\leq0$ for every $j\in[N]$. By the definition of
$\|\cdot\|_{*,\bm{\omega}}$,
\begin{align}
    \norm{\nabla H(\y)}_{*,\bm{\omega}}^2
    &\geq
    \sum_{j=1}^N
    \left[
        \frac{\omega_{j-1}^2}{\omega_j}
        \bigl(q(y_{j-1})p'(y_j)\bigr)^2
        +\omega_j(y_j^-)^2
    \right]\label{eq:y-negative-norm}\\
    &\geq
    \frac1N\sum_{j=1}^N
    \left[
        \bigl(q(y_{j-1})p'(y_j)\bigr)^2+(y_j^-)^2
    \right].
    \label{eq:dual-chain-gradient-energy}
\end{align}

We first establish a bound on the activation deficit $1-q(y_k)$ for each $k\in[N]$. If $q(y_k)<1$,
we choose the smallest $j\leq k$ such that
$1-q(y_j)\geq\frac12(1-q(y_k))$, which leads to $q(y_{j-1})>1-\frac12(1-q(y_k))\geq\frac12$. Then, we have 
\begin{equation}
    \label{eq:est-1}
    1-q(y_k)
    \leq2\bigl(1-q(y_j)\bigr)
    \leq4\left((p'(y_j))^2+(y_j^-)^2\right)
    \leq16N\norm{\nabla H(\y)}_{*,\bm{\omega}}^2,
\end{equation}
where the second inequality follows from
Lemma~\ref{lem:ncpl-gate-bounds}, and the last follows from
\eqref{eq:dual-chain-gradient-energy} and
$q(y_{j-1})\ge \tfrac12$.
If $q(y_k)=1$, the same bound holds trivially.

The preceding estimate also controls the gate term in each local interaction. For every $j\in[N]$,
\[
\begin{aligned}
    1-q(y_{j-1})p(y_j)
    &=
    1-q(y_{j-1})
    +q(y_{j-1})\bigl(1-p(y_j)\bigr)\\
    &\leq
    1-q(y_{j-1})+1-q(y_j)\\
    &\leq
    32N\norm{\nabla H(\y)}_{*,\bm{\omega}}^2,
\end{aligned}
\]
where the first inequality follows from
Lemma~\ref{lem:ncpl-gate-bounds}, and the second follows from \eqref{eq:est-1}.
Multiplying by $\omega_{j-1}$ and summing over $j$ gives
\[
\begin{aligned}
    \sum_{j=1}^N\omega_{j-1}
    \bigl(1-q(y_{j-1})p(y_j)\bigr)
    &\leq
    32N\norm{\nabla H(\y)}_{*,\bm{\omega}}^2
    \sum_{j=1}^N\omega_{j-1}<
    96N\norm{\nabla H(\y)}_{*,\bm{\omega}}^2,
\end{aligned}
\]
where the last inequality follows from
Lemma~\ref{lem:ncpl-weight-estimates}.

Moreover, \eqref{eq:y-negative-norm}  gives
$\frac12\norm{\y^-}_{\bm{\omega}}^2
\leq\frac12\norm{\nabla H(\y)}_{*,\bm{\omega}}^2$. 
Combining the preceding bounds yields
\[
    H(\y)+10\bigl(1-q(y_N)\bigr)
    <\left(256N+\frac12\right)
    \norm{\nabla H(\y)}_{*,\bm{\omega}}^2
    \leq320N\norm{\nabla H(\y)}_{*,\bm{\omega}}^2,
\]
where the last inequality uses $N\ge 2$.
\end{proof}

Proposition~\ref{lem:ncpl-inner-technical} therefore provides
a length-$N$ sequential dual chain with a weighted P\L{} modulus of order
$1/N$, together with control of its terminal activation deficit. We next
couple this chain to the outer primal variables. 


\subsection{A Terminal-Gated Primal-Dual Coupling}
\label{sec:inner}
The function $-H$ provides the required dual P\L{} geometry, but its maximum is fixed at zero. To couple the dual chain with the outer primal construction, 
let $\alpha\in[0,10]$ be a scalar parameter and define
\begin{equation*}G(\alpha;\y):=\alpha q(y_N)-H(\y).\end{equation*} 
Later, $\alpha$ will be instantiated by a scalar function $a(\x)$ of the primal variables. In that case,
\[\nabla_{\x}G(a(\x);\y)=q(y_N)\nabla a(\x),\]
so the corresponding primal gradient remains hidden whenever $y_N\leq0$.
The following lemma shows that $G(\alpha;\cdot)$ has maximum value
$\alpha$ and satisfies a weighted dual P\L{} inequality with modulus
$1/(640N)$.


\begin{lemma}
\label{lem:ncpl-inner-block}
For every $\alpha\in[0,10]$, we have 
\begin{equation}
\max_{\widetilde{\y}\in\R^N}G(\alpha;\widetilde{\y})=\alpha.
\label{eq:normalized-inner-value}
\end{equation}
Moreover, for every $\y\in\R^N$,
\begin{equation}
\frac12\norm{\nabla_{\y}G(\alpha;\y)}_{*,\bm{\omega}}^2
\geq
\frac{1}{640N}
\left(
\max_{\widetilde{\y}\in\R^N}G(\alpha;\widetilde{\y})
-G(\alpha;\y)
\right).
\label{eq:normalized-inner-pl}
\end{equation}
\end{lemma}

\begin{proof}[Proof of Lemma~\ref{lem:ncpl-inner-block}]
Since $0\leq q\leq1$, $H\geq0$, and $H(\mathbf{1})=0$, we have
$G(\alpha;\y)\leq\alpha$, with equality at
$\y=\mathbf{1}$. This proves \eqref{eq:normalized-inner-value}.

By Proposition~\ref{lem:ncpl-inner-technical} and $0\leq\alpha\leq10$,
\[
\alpha-G(\alpha;\y)
=
H(\y)+\alpha\bigl(1-q(y_N)\bigr)
\leq
320N\norm{\nabla H(\y)}_{*,\bm{\omega}}^2.
\]
Furthermore, $\nabla_{y_j}G=-\nabla_{y_j}H$ for $j<N$, while
$
\nabla_{y_N}G
=
\alpha q'(y_N)-\nabla_{y_N}H.
$
Since $\alpha q'(y_N)\geq0$ and $\nabla_{y_N}H\leq0$, it follows that
\[
\norm{\nabla H(\y)}_{*,\bm{\omega}}
\leq
\norm{\nabla_{\y}G(\alpha;\y)}_{*,\bm{\omega}}.
\]
Combining the preceding inequalities proves
\eqref{eq:normalized-inner-pl}.
\end{proof}


The function $G$ has the required primal-dual coupling and
weighted dual P\L{} properties, but it is expressed in a weighted
geometry and at a normalized scale. We next use a diagonal change of
coordinates to recover the Euclidean geometry and a perspective
transformation to match the scale of the outer construction.
Define
$
\bW:=
\operatorname{diag}
(\sqrt{\omega_1},\ldots,\sqrt{\omega_N}).
$ 
For $\eta\geq0$, $0\leq\alpha\leq10\eta^2$, and $\y\in\R^N$, define
\begin{equation*}
B(\eta,\alpha;\y)
:=
\begin{cases}
\displaystyle
\eta^2G
\left(\frac{\alpha}{\eta^2};
\frac{\bW^{-1}\y}{\eta}\right),
& \eta>0,\\[3mm]
    -\frac12\norm{\y^-}^2,
& \eta=0.
\end{cases}
\end{equation*}
For $\eta>0$, the dependence of $B$ on the scalar target is given by
\begin{equation}
\nabla_\alpha B(\eta,\alpha;\y)
=
q\left(\frac{y_N}{\eta}\right),
\label{eq:scaled-inner-interface}
\end{equation}
where we used $\omega_N=1$. Thus, the target remains hidden from the
primal gradient while the terminal dual coordinate is zero.
The following lemma shows that the coordinate normalization and perspective scaling preserve the exact maximization identity and the dual P\L{} modulus, with the latter now measured in the Euclidean norm.

\begin{lemma}
\label{lem:ncpl-scaled-inner-block}
For every $\eta\geq0$, $\alpha\in[0,10\eta^2]$, and $\y\in\R^N$,
\begin{equation*}
\max_{\widetilde{\y}\in\R^N}B(\eta,\alpha;\widetilde{\y})=\alpha.
\end{equation*}
Moreover,
\begin{equation}
\frac12\norm{\nabla_{\y}B(\eta,\alpha;\y)}^2
\geq
\frac{1}{640N}
\left(
\max_{\widetilde{\y}\in\R^N}B(\eta,\alpha;\widetilde{\y})
-B(\eta,\alpha;\y)
\right).
\label{eq:scaled-inner-pl}
\end{equation}
\end{lemma}

\begin{proof}[Proof of Lemma~\ref{lem:ncpl-scaled-inner-block}]
Suppose first that $\eta>0$, and set
$
\widetilde\alpha:=\frac{\alpha}{\eta^2},~
\widetilde{\y}:=\frac{\bW^{-1}\y}{\eta}.
$
Lemma~\ref{lem:ncpl-inner-block} gives
\[
\max_{\boldsymbol v}B(\eta,\alpha;\boldsymbol v)
=
\eta^2\max_{\boldsymbol u}G(\widetilde\alpha;\boldsymbol u)
=
\alpha.
\]
Moreover,
\[
\nabla_{\y}B(\eta,\alpha;\y)
=
\eta\bW^{-1}\nabla_{\widetilde{\y}}G(\widetilde\alpha;\widetilde{\y}),
\]
and hence
\[
\norm{\nabla_{\y}B(\eta,\alpha;\y)}^2
=
\eta^2\norm{\nabla_{\widetilde{\y}}G(\widetilde\alpha;\widetilde{\y})}_{*,\bm{\omega}}^2.
\]
The dual gap and squared gradient norm are both multiplied by
$\eta^2$. Hence, \eqref{eq:scaled-inner-pl} follows from
Lemma~\ref{lem:ncpl-inner-block}.

If $\eta=0$, then $\alpha=0$ and
$B(0,0;\y)=-\frac12\norm{\y^-}^2$. Thus its maximum is zero, and
\[
\max_{\widetilde{\y}}B(0,0;\widetilde{\y})-B(0,0;\y)
=
\frac12\norm{\y^-}^2,
\qquad\text{and}\qquad
\norm{\nabla_{\y}B(0,0;\y)}^2
=
\norm{\y^-}^2,
\]
which also proves \eqref{eq:scaled-inner-pl}.
\end{proof}
\subsection{The Composite NC-P\L{} Hard Instance}
\label{sec:outer}
We now embed the scaled dual block \(B\) constructed in
Section~\ref{sec:inner} into each link of the nonconvex zero-chain of
\citet{carmon2019lower}. Maximizing over the dual variables recovers this
zero-chain exactly as the value function, while each embedded block
inserts \(N\) dual coordinates between two consecutive primal coordinates.



We first recall the scalar functions used in this zero-chain
construction. Define
$$
\psi(t):=
\begin{cases}
0, & t\leq\frac12,\\
\displaystyle
\exp\left(1-\frac{1}{(2t-1)^2}\right), & t>\frac12,
\end{cases}
\quad \text{and} \quad
\varphi(t):=\sqrt e\int_{-\infty}^t e^{-\frac{s^2}{2}}\,\diff s.
$$


We use the following properties of $\psi$ and $\varphi$,
established by \citet[Lemma~1]{carmon2019lower}.
\begin{fact}
\label{lem:outer-activation-functions}
The functions $\psi$ and $\varphi$ satisfy the following properties.
\begin{enumerate}
\item[(i)]
The functions $\psi$ and $\varphi$ are infinitely differentiable.
\item[(ii)]
For every integer $k\geq0$ and every $t\leq\frac12$,
\(
\psi^{(k)}(t)=0.
\)
\item[(iii)]
For every $s\geq1$ and $|t|<1$,
\(
\psi(s)\varphi'(t)>1.
\)
\item[(iv)]
The functions $\psi$, $\psi'$, $\varphi$, and $\varphi'$ are nonnegative and
bounded. More precisely,
\[
0\leq\psi(t)<e,\quad
0\leq\psi'(t)\leq\sqrt{\frac{54}{e}},\quad
0<\varphi(t)<\sqrt{2\pi e},
\quad\mbox{and}\quad
0<\varphi'(t)\leq\sqrt e.
\]
\end{enumerate}
\end{fact}

Using these scalar functions, define the local interaction
\begin{equation*}
h(s,t)
:=
\psi(-s)\varphi(-t)-\psi(s)\varphi(t).
\end{equation*}
With \(x_0:=1\), summing \(h(x_{i-1},x_i)\) over \(i\in[M]\)
gives the length-\(M\) nonconvex zero-chain of
\citet{carmon2019lower}. We will represent each interaction
\(h(x_{i-1},x_i)\) as the maximum of a primal-dual block over an
associated \(N\)-dimensional dual variable.

To construct this block from $B$, define
\[
\rho(s):=\sqrt{\psi(s)+\psi(-s)}
\qquad\text{and}\qquad
\widetilde h(s,t):=5\rho(s)^2+h(s,t).
\]
The following two lemmas establish the regularity of $\rho$ and
verify that $(\eta,\alpha)=(\rho(s),\widetilde h(s,t))$ lies in the
admissible parameter range of $B$.




\begin{lemma}
\label{lem:ncpl-outer-scale-bounds}
The function $\rho$ is nonnegative and smooth on $\R$.
Moreover, $\rho$ and $\rho'$ are bounded by numerical constants,
$\rho'$ is Lipschitz, and
\[
\rho(s)=\rho'(s)=0
\qquad\text{whenever}\qquad
|s|\leq\frac12.
\]
\end{lemma}

\begin{lemma}
\label{lem:ncpl-lifted-target-bounds}
For every $t\in\R$,
\[
0<5-\varphi(t)<10,
\qquad
0<5+\varphi(-t)<10,
\quad \text{and} \quad 
|\varphi'(t)|+\Lip(\varphi')<3.
\]
Moreover, we have 
\begin{equation}
\widetilde h(s,t)
=
\begin{cases}
\rho(s)^2\bigl(5-\varphi(t)\bigr), & s>\frac12,\\
0, & |s|\leq\frac12,\\
\rho(s)^2\bigl(5+\varphi(-t)\bigr), & s<-\frac12.
\end{cases}
\label{eq:lifted-target-active-form}
\end{equation}
\end{lemma}
The proofs of Lemmas~\ref{lem:ncpl-outer-scale-bounds}
and~\ref{lem:ncpl-lifted-target-bounds} are given in
Appendices~\ref{app:proof-ncpl-outer-scale-bounds}
and~\ref{app:proof-ncpl-lifted-target-bounds}, respectively.

We can now realize the interaction $h$ through the block $B$. For
$s,t\in\R$ and $\y\in\R^N$, define the embedded
primal-dual block
\begin{equation*}
C(s,t;\y)
:=
-5\rho(s)^2
+
B\bigl(\rho(s),\widetilde h(s,t);\y\bigr).
\end{equation*}
By Lemma~\ref{lem:ncpl-lifted-target-bounds}, the arguments supplied to
$B$ are admissible, i.e., $
0\leq\widetilde h(s,t)\leq10\rho(s)^2
$. Therefore, Lemma~\ref{lem:ncpl-scaled-inner-block} gives
\begin{equation}
\max_{\y\in\R^N}C(s,t;\y)
=
-5\rho(s)^2+\widetilde h(s,t)
=
h(s,t).
\label{eq:embedded-outer-component}
\end{equation}
Thus, maximizing over the dual variables recovers the local interaction
$h(s,t)$ exactly.

The derivative of the embedded block with respect to the successor primal
variable is controlled by the terminal dual coordinate. If $\rho(s)>0$,
then the chain rule, together with
$\nabla_t\widetilde h(s,t)=\nabla_t h(s,t)$, gives
\begin{equation}
\nabla_t C(s,t;\y)
=
q\left(\frac{y_N}{\rho(s)}\right)
\nabla_t h(s,t).
\label{eq:embedded-successor-interface}
\end{equation}
Thus, $\nabla_tC(s,t;\y)=0$ whenever $y_N=0$. If $\rho(s)=0$, then
\[
C(s,t;\y)
=
B(0,0;\y)
=
-\frac12\norm{\y^-}^2,
\]
which is independent of $t$. Therefore, the block produces no gradient in
the successor primal variable while its terminal dual coordinate remains
zero.

We now place one copy of the block \(C\) on each of the \(M\) links of
the outer zero-chain. For every \(i\in[M]\), introduce a dual block  
$
\y^{(i)}
=
\bigl(y^{(i)}_1,\ldots,y^{(i)}_N\bigr)
\in\R^N,
$ 
and write
$
\y
=
\bigl(\y^{(1)},\ldots,\y^{(M)}\bigr)
\in\R^{MN}.
$
Setting \(x_0:=1\), define the unscaled saddle function
\begin{equation*}
\bar f(\x;\y)
:=
\sum_{i=1}^M
C\bigl(x_{i-1},x_i;\y^{(i)}\bigr).
\end{equation*}
Because the dual blocks are independent, maximization over \(\y\)
separates across \(i\). Consequently,
\eqref{eq:embedded-outer-component} gives
\[
\bar\Phi(\x)
:=
\max_{\y\in\R^{MN}}\bar f(\x;\y)
=
\sum_{i=1}^M h(x_{i-1},x_i).
\]
Thus, dual maximization preserves the nonconvex zero-chain of
\citet{carmon2019lower} as the value function, whereas the terminal gates
modify its first-order information structure: each transition from
$x_{i-1}$ to $x_i$ requires the sequential discovery of the $N$
coordinates in $\y^{(i)}$.

Define $\ord=(\pi_1,\ldots,\pi_{M(N+1)})$ by
\begin{equation}
\pi_{(i-1)(N+1)+j}
:=
\begin{cases}
M+(i-1)N+j,
    & j\in[N],\\
i,
    & j=N+1,
\end{cases}
\qquad i\in[M].
\label{eq:joint-coordinate-order}
\end{equation}
Equivalently, $\ord$ orders the coordinates as
\[
\bigl(
y^{(1)}_1,\ldots,y^{(1)}_N,x_1,
\ldots,
y^{(M)}_1,\ldots,y^{(M)}_N,x_M
\bigr).
\]
The sequential structure of each dual chain prevents
$y^{(i)}_j$ from being revealed before $y^{(i)}_{j-1}$. Moreover,
\eqref{eq:embedded-successor-interface} prevents $x_i$ from being revealed
before $y^{(i)}_N$, while the flatness of the outer scale keeps the next
block inactive until $x_i$ is revealed. The resulting discovery order is
therefore
\[
x_{i-1}
\longrightarrow
y^{(i)}_1
\longrightarrow\cdots\longrightarrow
y^{(i)}_N
\longrightarrow
x_i.
\]
In particular, $x_M$ is the final coordinate under $\ord$, and hence
$
\pi_{M(N+1)}=M.
$ 

We next verify that $(\bar f,\ord)$ satisfies the four hard-instance
certificates in Condition~\ref{cond:normalized-hard-instance}.

\section{Verification of the Hard Instance Properties}
\label{sec:verification}

We now verify the four certificates in
Condition~\ref{cond:normalized-hard-instance} and then scale the resulting
instance to prove the main theorems. 

\subsection{The Function Class Certificate}
\label{sec:verification-membership}
We verify
Condition~\ref{cond:normalized-hard-instance}~\textup{(C1)} by
establishing uniform joint smoothness of $\bar f$ and a dual P\L{} modulus
of order $1/N$. The only nonroutine smoothness issue is the extension of
the perspective block through the degenerate regime $\rho(s)=0$; the
following lemma records the required result.
\begin{lemma}[Uniform smoothness of the embedded block]
\label{lem:ncpl-joint-smoothness}
There exists a numerical constant \(L_C>0\), independent of \(N\), such
that \(C\) is \(L_C\)-smooth on \(\R^{N+2}\).
\end{lemma}
The proof is given in
Appendix~\ref{app:proof-ncpl-joint-smoothness}. 



\begin{proposition}[Smoothness and dual P\L{} certificate]
\label{prop:ncpl-function-class}
For every $\x\in\R^M$, the maximum defining $\bar\Phi(\x)$
is finite and attained, and
\[
\bar\Phi(\x)=\sum_{i=1}^M h(x_{i-1},x_i).
\]
Moreover, there exist numerical constants $\ell_0,\mu_0>0$, independent
of $M$ and $N$, such that $\bar f$ is jointly $\ell_0$-smooth and 
\begin{equation}
\frac12\norm{\nabla_{\y}\bar f(\x;\y)}^2
\geq
\frac{\mu_0}{N}
\bigl(\bar\Phi(\x)-\bar f(\x;\y)\bigr).
\label{eq:hard-instance-dual-pl}
\end{equation}
Consequently,
Condition~\ref{cond:normalized-hard-instance}~\textup{(C1)} holds.
\end{proposition}
\begin{proof}[Proof of Proposition~\ref{prop:ncpl-function-class}]
Equation~\eqref{eq:embedded-outer-component} shows that the
maximum of each block over $\y^{(i)}$ is finite, attained, and equal to
$h(x_{i-1},x_i)$. Because the dual blocks are disjoint, their maximizers
can be chosen independently. This proves the stated identity for
$\bar\Phi$.

For \(i\in[M]\), write
$C_i(\x;\y):=C(x_{i-1},x_i;\y^{(i)}).
$ 
By Lemma~\ref{lem:ncpl-joint-smoothness}, each \(C_i\) is
\(L_C\)-smooth and depends only on
\((x_{i-1},x_i,\y^{(i)})\). Hence two blocks have disjoint variable
supports whenever their indices differ by at least two. In particular,
the odd-indexed blocks have pairwise disjoint supports, and the same is
true of the even-indexed blocks.

Because the blocks within each group act on mutually orthogonal
coordinate subspaces, their smoothness constants do not add. Therefore,
both
\[
\sum_{\substack{i\in[M]\\ i\ {\rm odd}}}C_i
\qquad\text{and}\qquad
\sum_{\substack{i\in[M]\\ i\ {\rm even}}}C_i
\]
are \(L_C\)-smooth. Since \(\bar f\) is the sum of these two functions,
$
\Lip(\nabla\bar f)\leq L_C+L_C=2L_C.
$
Thus, \(\bar f\) is jointly \(\ell_0\)-smooth with
$
\ell_0:=2L_C.
$ 

We next establish the dual P\L{} inequality. The additive term
\(-5\rho(x_{i-1})^2\) in the definition of \(C\) is independent of
\(\y^{(i)}\). Hence, by \eqref{eq:scaled-inner-pl} and
\eqref{eq:embedded-outer-component}, 
\[
\frac12
\norm{
\nabla_{\y^{(i)}}
C(x_{i-1},x_i;\y^{(i)})
}^2
\geq
\frac{1}{640N}
\left(
h(x_{i-1},x_i)
-
C(x_{i-1},x_i;\y^{(i)})
\right).
\]
Summing over \(i\in[M]\), using the disjointness of the dual blocks and
the identity
$
\bar\Phi(\x)=\sum_{i=1}^M h(x_{i-1},x_i),
$
gives
\[
\begin{aligned}
\frac12\norm{\nabla_{\y}\bar f(\x;\y)}^2
&=
\sum_{i=1}^M
\frac12
\norm{
\nabla_{\y^{(i)}}
C(x_{i-1},x_i;\y^{(i)})
}^2\\
&\geq
\frac{1}{640N}
\sum_{i=1}^M
\left(
h(x_{i-1},x_i)
-
C(x_{i-1},x_i;\y^{(i)})
\right)\\
&=
\frac{1}{640N}
\bigl(\bar\Phi(\x)-\bar f(\x;\y)\bigr).
\end{aligned}
\]
Thus, \eqref{eq:hard-instance-dual-pl} holds with
\(\mu_0:=1/640\). 
\end{proof}

\subsection{The Sequential-Discovery Certificate}
\label{sec:verification-zero-chain}

We next verify
Condition~\ref{cond:normalized-hard-instance}~\textup{(C2)}. We first
record the behavior of an embedded block in its inactive regime.
If \(\lvert s\rvert\leq\frac12\), then
\(\psi(s)=\psi(-s)=0\), and hence
\(\rho(s)=\widetilde h(s,t)=0\). The definition of $C$
therefore gives $C(s,t;\y)=-\frac12\norm{\y^-}^2$. Differentiating this
expression gives
\begin{equation}
\nabla_s C(s,t;\y)
=
\nabla_t C(s,t;\y)
=
0,
\quad \text{and} \quad
\nabla_{\y}C(s,t;\y)
=
\y^-.
\label{eq:embedded-block-inactive}
\end{equation}

\begin{proposition}[Saddle zero-chain]
\label{prop:ncpl-saddle-zero-chain}
Let $\ord$ be defined in
\eqref{eq:joint-coordinate-order}.  For every $k\in[M(N+1)]$,
$$
    \supp_{\ord}(\x,\y)\subseteq[k-1]
    \quad\Longrightarrow\quad
    \supp_{\ord}\bigl(
        \nabla_{\x}\bar f(\x;\y),
        \nabla_{\y}\bar f(\x;\y)
    \bigr)
    \subseteq[k].
$$
Consequently, $\bar f$ is a first-order saddle zero-chain with respect to
$\ord$.
\end{proposition}

\begin{proof}[Proof of Proposition~\ref{prop:ncpl-saddle-zero-chain}]

Fix \(k\in[M(N+1)]\) and suppose that
$
\supp_{\ord}(\x,\y)\subseteq[k-1].
$ 
We distinguish whether the \(k\)-th coordinate under \(\ord\) is dual or
primal.

Suppose first that \(\pi_k\) corresponds to \(y_j^{(i)}\). Then 
\begin{equation}
y_j^{(i)}=\cdots=y_N^{(i)}=0,
\qquad
\y^{(i+1)}=\cdots=\y^{(M)}=\bz_N,
\qquad\text{and}\qquad
x_i=\cdots=x_M=0.
\label{eq:zero-chain-dual-case}
\end{equation}
All blocks with index smaller than \(i\) involve only coordinates
preceding \(y_j^{(i)}\). Moreover, every block with index greater than
\(i\) is evaluated at
$
(x_{r-1},x_r;\y^{(r)})=(0,0;\bz_N).
$ Its full gradient therefore vanishes by
\eqref{eq:embedded-block-inactive}. It remains to examine the \(i\)-th
block.

Set $
\eta:=\rho(x_{i-1}).
$
If \(\eta>0\), define
\[
\widetilde{\y}
:=
\frac{\bW^{-1}\y^{(i)}}{\eta},
\quad \text{and} \quad \widetilde\alpha
:=
\frac{\widetilde h(x_{i-1},x_i)}{\eta^2}.
\]
The diagonal transformation preserves coordinate supports, and
\begin{equation}
\nabla_{y_r^{(i)}}C(x_{i-1},x_i;\y^{(i)})
=
\frac{\eta}{\sqrt{\omega_r}}
\left(
-\nabla_{\widetilde y_r}H(\widetilde{\y})
+
\mathbf{1}_{\{r=N\}}
\widetilde\alpha q'(\widetilde y_N)
\right).
\label{eq:gradient-C}
\end{equation}
For every \(r\in\{j+1,\ldots,N\}\),
\eqref{eq:zero-chain-dual-case} implies that all coordinates of
\(\widetilde{\y}\) entering
\(\nabla_{\widetilde y_r}H(\widetilde{\y})\) are zero.
 Hence,
\eqref{eq:gradient-H} and \eqref{eq:gradient-C}, together with
\(q(0)=q'(0)=0\) and \(0^-=0\), give
\[
\nabla_{y_r^{(i)}}C(x_{i-1},x_i;\y^{(i)})
=
0,
\qquad r>j.
\] Furthermore, \(y_N^{(i)}=0\), so
\eqref{eq:embedded-successor-interface} yields
$
\nabla_tC(x_{i-1},x_i;\y^{(i)})=0.
$

If \(\eta=0\), then
\(\lvert x_{i-1}\rvert\leq\frac12\). Therefore,
\eqref{eq:embedded-block-inactive} gives
\[
\nabla_tC(x_{i-1},x_i;\y^{(i)})=0,
\quad \text{and} \quad
\nabla_{\y}C(x_{i-1},x_i;\y^{(i)})
=
(\y^{(i)})^-.
\]
Since \(y_j^{(i)}=\cdots=y_N^{(i)}=0\), the latter vector is supported
only on \(y_1^{(i)},\ldots,y_{j-1}^{(i)}\). Thus, in either case, the
\(i\)-th block produces no gradient in coordinates following
\(y_j^{(i)}\). Together with the vanishing gradients of the later blocks,
this proves
$\supp_{\ord}\bigl(
\nabla_{\x}\bar f(\x;\y),
\nabla_{\y}\bar f(\x;\y)
\bigr)
\subseteq[k].
$

Suppose next that \(\pi_k\) corresponds to \(x_i\). Then
\[
x_i=\cdots=x_M=0,
\quad \text{and} \quad
\y^{(i+1)}=\cdots=\y^{(M)}=\bz_N.
\]
All variables on which the \(i\)-th block depends,
namely \(x_{i-1}\), \(\y^{(i)}\), and \(x_i\), lie among the first \(k\)
coordinates under \(\ord\). Every block with index greater than \(i\) is
again evaluated at \((0,0;\bz_N)\), so its full gradient vanishes by
\eqref{eq:embedded-block-inactive}. Hence,
$
\supp_{\ord}\bigl(
\nabla_{\x}\bar f(\x;\y),
\nabla_{\y}\bar f(\x;\y)
\bigr)
\subseteq[k].
$
This completes the proof.

\end{proof}

\subsection{Terminal Obstruction and Initial Gap}
\label{sec:verification-value}


We now establish the terminal-gradient and initial-gap certificates for
this explicit value function.

\begin{proposition}[Terminal obstruction and initial gap]
\label{prop:ncpl-terminal-gap}
The value function $\bar\Phi$ is differentiable. The last position in
$\ord$ corresponds to $x_M$. Moreover,
\[
x_M=0
\quad\Longrightarrow\quad
\norm{\nabla\bar\Phi(\x)}\geq1,
\]
and
\[
\bar\Phi(\bz)-\inf_{\x\in\R^M}\bar\Phi(\x)\leq12M.
\]
Thus, Conditions~\ref{cond:normalized-hard-instance}~\textup{(C3)}
and~\textup{(C4)} hold with $g_0:=1$ and $\Delta_0:=12$.
\end{proposition}

\begin{proof}[Proof of Proposition~\ref{prop:ncpl-terminal-gap}]
Since $\psi$ and $\varphi$ are smooth, so are $h$ and $\bar\Phi$. Moreover,
\eqref{eq:joint-coordinate-order} gives $\pi_{M(N+1)}=M$, so the last
position in $\ord$ corresponds to $x_M$.

Direct differentiation, together with the nonnegativity properties in
Fact~\ref{lem:outer-activation-functions}~\textup{(iv)}, gives
\begin{equation}
    \label{eq:1}
    \nabla_s h(s,t)
    =
    -\psi'(-s)\varphi(-t)-\psi'(s)\varphi(t)
    \leq0.
\end{equation}
Moreover, since $\varphi'$ is even,
\begin{equation}
    \label{eq:2}
    \nabla_t h(s,t)
    =
    -\bigl(\psi(s)+\psi(-s)\bigr)\varphi'(t)
    =
    -\rho(s)^2\varphi'(t)
    \leq0.
\end{equation}
Suppose that $x_M=0$, and let $k\in[M]$ be the smallest index such that
$|x_k|<1$. Such an index exists, and $|x_{k-1}|\geq1$ because $x_0=1$. We thus obtain
\begin{align*}
    -\nabla_{x_k}\bar\Phi(\x)
    &=
    \begin{cases}
    -\nabla_t h(x_{k-1},x_k)
    -\nabla_s h(x_k,x_{k+1}),
    & k<M,\\[1mm]
    -\nabla_t h(x_{M-1},x_M),
    & k=M,
    \end{cases}\\
    &\geq
    -\nabla_t h(x_{k-1},x_k)
    =
    \rho(x_{k-1})^2\varphi'(x_k)
    >1,
\end{align*}
where the first inequality follows from \eqref{eq:1} and the
second equality from \eqref{eq:2}. For the final inequality, note that $|x_{k-1}|\geq1$ and
$|x_k|<1$. Since
$\rho(x_{k-1})^2\geq\psi(|x_{k-1}|)$,
Fact~\ref{lem:outer-activation-functions}~\textup{(iii)},
applied with $s=|x_{k-1}|$ and $t=x_k$, yields the claim.
Thus $\norm{\nabla\bar\Phi(\x)}>1$, which proves the terminal bound.

Finally, Fact~\ref{lem:outer-activation-functions}~\textup{(iv)} gives
\[
    h(s,t)
    =
    \psi(-s)\varphi(-t)-\psi(s)\varphi(t)
    \geq
    -\psi(s)\varphi(t)
    >
    -e\sqrt{2\pi e}
    >
    -12.
\]
Consequently,
\(
    \inf_{\x\in\R^M}\bar\Phi(\x)\geq-12M.
\) Since $x_0=1$, we have
\[
    \bar\Phi(\bz)
    =
    h(1,0)+\sum_{i=2}^Mh(0,0)
    =
    -\varphi(0)
    <0.
\]
It follows that
\[
    \bar\Phi(\bz)-\inf_{\x\in\R^M}\bar\Phi(\x)
    \leq12M.
\]
This proves Conditions~\ref{cond:normalized-hard-instance}~\textup{(C3)}
and~\textup{(C4)}.
\end{proof}
\subsection{Proofs of Theorems~\ref{thm:zr-lower} and
\ref{thm:deterministic}}
\label{sec:pf-thm}
We first use the following scaling lemma to extend the unscaled hard instance to general parameter settings.
\begin{lemma}[Scaling Lemma]
\label{lem:scaling}
Fix an admissible pair \((M,N)\), and assume
Condition~\ref{cond:normalized-hard-instance}. For \(\ell,\lambda>0\), set
\[
    f(\x;\y)
    :=
    \frac{\ell\lambda^2}{\ell_0}
    \bar f\!\left(
        \frac{\x}{\lambda};
        \frac{\y}{\lambda}
    \right)
\]
on \(\R^M\times\R^{MN}\). Then the following properties hold:
\begin{enumerate}[(i)]
    \item For every $\x$, the maximum defining $\Phi(\x)$ is finite and attained, and $\Phi$ satisfies 
    \begin{align}
        \Phi(\x)
        &=
        \frac{\ell\lambda^2}{\ell_0}
        \bar\Phi\!\left(\frac{\x}{\lambda}\right),
        \qquad\text{and}\qquad
        \nabla\Phi(\x)
        =
        \frac{\ell\lambda}{\ell_0}
        \nabla\bar\Phi\!\left(\frac{\x}{\lambda}\right).
        \label{eq:gradient-scaling}
    \end{align}

    \item The function \(f\) is jointly \(\ell\)-smooth and
    \(\frac{\ell\mu_0}{\ell_0N}\)-dual P\L{}.

    \item The scaling preserves coordinate supports and hence the saddle
    zero-chain property with respect to the same coordinate ordering.
\end{enumerate}
\end{lemma}

\begin{proof}[Proof of Lemma~\ref{lem:scaling}]
For \textup{(i)}, the map
\(\y\mapsto\frac{\y}{\lambda}\) is a bijection on \(\R^{MN}\), so it
preserves finiteness and attainment of the inner maximum and gives the value
function identity in \eqref{eq:gradient-scaling}; differentiating gives the
gradient identity.

For \textup{(ii)}, the chain rule gives
\[
    \nabla f(\x;\y)
    =
    \frac{\ell\lambda}{\ell_0}
    \nabla\bar f\!\left(
        \frac{\x}{\lambda};
        \frac{\y}{\lambda}
    \right),
\]
which proves joint smoothness immediately.
Condition~\ref{cond:normalized-hard-instance}~\textup{(C1)} then yields
\[
    \frac{1}{2}\norm{\nabla_{\y}f(\x;\y)}^2
    \geq
    \frac{\ell\mu_0}{\ell_0N}
    \bigl(\Phi(\x)-f(\x;\y)\bigr).
\]

Finally, for \textup{(iii)}, since the scaling factors are positive,
\[
    \supp(\x,\y)
    =
    \supp\left(
        \frac{\x}{\lambda},
        \frac{\y}{\lambda}
    \right),
    \qquad\mbox{and}\qquad
    \supp\nabla f(\x;\y)
    =
    \supp\nabla\bar f\left(
        \frac{\x}{\lambda};
        \frac{\y}{\lambda}
    \right).
\]
Thus, the scaling preserves the saddle zero-chain property under the same
coordinate ordering.
\end{proof}
We now combine the four hard-instance certificates with the
scaling identities to prove Theorems~\ref{thm:zr-lower}
and~\ref{thm:deterministic}. 
\begin{proof}[Proof of Theorem~\ref{thm:zr-lower}]
Propositions~\ref{prop:ncpl-function-class},
\ref{prop:ncpl-saddle-zero-chain}, and~\ref{prop:ncpl-terminal-gap} verify
Condition~\ref{cond:normalized-hard-instance} with numerical constants
\(\ell_0,\mu_0,g_0,\Delta_0>0\). Set
\[
    c_0:=\frac{2\ell_0}{\mu_0},
    \qquad
    c_1:=\frac{g_0^2}{8\ell_0\Delta_0},
    \qquad\text{and}\qquad
    c_2:=\frac{g_0^2\mu_0}{8\ell_0^2\Delta_0}.
\]
Fix \(\ell,\mu,\Delta,\epsilon>0\) satisfying the assumptions of the
theorem, and choose
\begin{equation}
    \label{eq:parameter}
    N:=\left\lfloor\frac{\mu_0\kappa}{\ell_0}\right\rfloor,
    \qquad
    \lambda:=\frac{2\ell_0\epsilon}{\ell g_0},
    \qquad\text{and}\qquad
    M:=\left\lfloor
        \frac{\ell_0\Delta}{\ell\Delta_0\lambda^2}
    \right\rfloor.
\end{equation}
The assumptions $\kappa\ge c_0$ and
$\epsilon^2\le c_1\ell\Delta$ imply $M,N\ge2$. Thus, $(M,N)$ is
admissible.

Let \(\bar f\) be the resulting unscaled instance and define
\(
    f(\x;\y)
    :=
    \frac{\ell\lambda^2}{\ell_0}
    \bar f\left(\frac{\x}{\lambda};\frac{\y}{\lambda}\right).
\) We first verify that \(f\in\cF(\ell,\mu,\Delta)\).
Lemma~\ref{lem:scaling} shows that \(f\) is jointly \(\ell\)-smooth and its dual P\L{} constant is at least
\[
    \frac{\ell\mu_0}{\ell_0N}
    \geq\frac{\ell}{\kappa}
    =\mu.
\]
Moreover, Condition~\ref{cond:normalized-hard-instance}~\textup{(C4)} gives
\[
    \Phi(\bz)-\inf\Phi
    \leq
    \frac{\ell\lambda^2}{\ell_0}\Delta_0M
    \leq\Delta.
\]
Thus \(f\in\cF(\ell,\mu,\Delta)\).

We now show that every deterministic zero-respecting algorithm requires at
least \(M(N+1)\) queries to find an \(\epsilon\)-stationary point of \(f\). Let \(\sZ\in\cA_{\zr}\cap\cA_{\det}\) run on \(f\). Positive
scaling preserves coordinate supports, so Lemma~\ref{lem:scaling}~\textup{(iii)}
and Condition~\ref{cond:normalized-hard-instance}~\textup{(C2)} imply, by
induction, that
\[
    \supp_{\ord}\bigl(
        \sZ_{\x}^{(t)}[f],\sZ_{\y}^{(t)}[f]
    \bigr)
    \subseteq[t],
    \qquad t\in[M(N+1)].
\]
Because \(\pi_{M(N+1)}=M\), the last primal coordinate remains zero over
$t\in[M(N+1)-1]$. Hence Condition~\ref{cond:normalized-hard-instance}~\textup{(C3)}
and \eqref{eq:gradient-scaling} give, for every \(t<M(N+1)\),
\[
    \norm{\nabla\Phi(\sZ_{\x}^{(t)}[f])}
    =
    \frac{\ell\lambda}{\ell_0}
    \norm{\nabla\bar\Phi\left(\frac{\sZ_{\x}^{(t)}[f]}{\lambda}\right)}
    \geq
    \frac{\ell\lambda g_0}{\ell_0}
    =2\epsilon.
\]
The parameter selection in \eqref{eq:parameter} gives
\[
    M\geq
    \frac{g_0^2\ell\Delta}{8\ell_0\Delta_0\epsilon^2},
    \qquad\mbox{and}\qquad
    N+1\geq\frac{\mu_0\kappa}{\ell_0},
\]
where we use the fact $\lfloor u\rfloor\ge \max\{u/2,u-1\}$ for $u\ge 2$.
Therefore, every such \(\sZ\) requires at least
\[
    M(N+1)
    \geq
    c_2\frac{\kappa\ell\Delta}{\epsilon^2}
\]
queries to reach an \(\epsilon\)-stationary point.
\end{proof}

\begin{proof}[Proof of Theorem~\ref{thm:deterministic}]
Use the constants from Theorem~\ref{thm:zr-lower}, and fix parameters
satisfying its assumptions. Set
\[
    T_0
    :=
    \left\lceil
        c_2\frac{\kappa\ell\Delta}{\epsilon^2}
    \right\rceil.
\]
Theorem~\ref{thm:zr-lower} provides an instance
\(f\in\cF(\ell,\mu,\Delta)\), for which
every \(\sZ\in\cA_{\zr}\cap\cA_{\det}\) satisfies
\[
    \norm{\nabla\Phi(\sZ_{\x}^{(t)}[f])}>\epsilon,
    \qquad 0\leq t<T_0.
\]

Fix \(\sA\in\cA_{\det}\). Lemma~\ref{lem:finite-horizon-resisting-oracle}
provides \(\sZ\in\cA_{\zr}\cap\cA_{\det}\) and orthogonal embeddings $\bU$ and $\bV$ such that the first $T_0$
query transcripts of $\sA$ and $\sZ$ agree after projection. Lemma~\ref{lem:isometric-invariance} gives
\(f_{\bU,\bV}\in\cF(\ell,\mu,\Delta)\) and, for every \(t<T_0\),
\[
\begin{aligned}
    \norm{\nabla\Phi_{\bU,\bV}(
        \sA_{\x}^{(t)}[f_{\bU,\bV}])}
    &=
    \norm{\nabla\Phi(
        \bU^\top\sA_{\x}^{(t)}[f_{\bU,\bV}])}=
    \norm{\nabla\Phi(\sZ_{\x}^{(t)}[f])}
    >\epsilon.
\end{aligned}
\]
Thus \(\sA\) needs at least \(T_0\) queries on an instance in the target
class. Since \(\sA\) was arbitrary,
\[
    \cT_{\epsilon}(\cA_{\det},\cF(\ell,\mu,\Delta))
    \geq T_0
    \geq c_2\frac{\kappa\ell\Delta}{\epsilon^2}.
\]
This proves the theorem.
\end{proof}

\section{Closing Remarks}
\label{sec:close}
We established the deterministic first-order lower bound
$\Omega(\ell\Delta\kappa/\epsilon^2)$ for smooth NC-P\L{} minimax
optimization. Together with the matching upper bound of \citet{yang2022},
this result determines the optimal dependence on
$(\ell,\Delta,\kappa,\epsilon)$ in the parameter regime considered here
and distinguishes the linear $\kappa$ dependence of the NC-P\L{}
class from the $\sqrt{\kappa}$ dependence available under strong
concavity.  Two directions remain open. First, extending the construction to randomized first-order
methods requires a robust saddle zero-chain; the probabilistic zero-chain mechanism of
\citet{arjevani2023lower} is the natural starting point. 
Second, the two-sided P\L{} setting of \citet{yang2020global}, in which the primal
variable also satisfies a P\L{} condition, has a different upper bound landscape and is
not covered by our construction.

\section*{AI Disclosure}
The hard instance construction was developed through iterative collaboration first with \texttt{GPT-5.5 Pro}. Subsequently, \texttt{GPT-5.6 Sol Ultra} was used to simplify the
construction and improve the presentation of the proofs. The authors
independently verified all mathematical arguments and take full responsibility
for the content of the paper. An accompanying Lean formalization, developed with Codex and available at \url{https://github.com/SiyuPan04/Lower-Bounds-for-Nonconvex-PL-Minimax-Optimization}, provides a formal verification of the deterministic first-order oracle lower bound established in this paper.

\section*{Acknowledgements}
Jiajin Li was supported by the Natural Sciences and Engineering
Research Council of Canada through Discovery Grant RGPIN-2025-05817.
\bibliography{ref}
\bibliographystyle{abbrvnat}

\appendix
\section{Proofs of Technical Lemmas}
\label{app:ncpl-deferred-estimates}
This appendix proves the technical lemmas used in the main
text. Second-derivative bounds for piecewise-smooth functions are
understood almost everywhere and therefore bound the Lipschitz constants
of the corresponding first derivatives. Unless stated otherwise, all
unnamed constants are positive numerical constants independent of
$M,N,\ell,\mu,\Delta$, and $\epsilon$, and all big-$\cO$ bounds are
uniform in the variables under consideration and, where relevant, in
$N$.

{
\subsection{Proof of Lemma~\ref{lem:ncpl-gate-bounds}}
\label{app:proof-ncpl-gate-bounds}
}

\begin{proof}[Proof of Lemma~\ref{lem:ncpl-gate-bounds}]
Direct differentiation on the nonconstant polynomial pieces shows that
$q$ is nondecreasing from $0$ to $1$, with
$0\leq q'\leq\frac32$ and $|q''|\leq6$. Since $q'$ is
continuous across the junctions and $q$ is constant outside $[0,1]$,
the asserted bounds for $q$ follow.

Since $p(t)=q((t+1)/2)$, the chain rule gives
$0\leq p'\leq\frac34$ and
$\Lip(p')\leq\frac32$. The range bound for $p$ follows from that
of $q$. Moreover, $(t+1)/2\geq t$ for $t\leq1$, whereas both
$p(t)$ and $q(t)$ equal one for $t\geq1$; hence
$p(t)\geq q(t)$.

It remains to verify the final inequality. For $0\leq t\leq1$,
\[
1-q(t)=(1-t)^2(1+2t),
\qquad 
p'(t)=\frac34(1-t^2),
\]
and $1+2t\leq\frac98(1+t)^2$. For $-1\leq t\leq0$, setting
$u=t^2$ gives
\[
2\left((p'(t))^2+(t^-)^2\right)-1
=
\frac{1-2u+9u^2}{8}
\geq0.
\]
The cases $t\leq-1$ and $t\geq1$ follow directly from
$t^-= -t$ and $q(t)=1$, respectively. This completes the proof.
\end{proof}

{
\subsection{Proof of Lemma~\ref{lem:ncpl-weight-estimates}}
\label{app:proof-ncpl-weight-estimates}
}

\begin{proof}[Proof of Lemma~\ref{lem:ncpl-weight-estimates}]
By the definition of the weights,
$\omega_{N-1}=\omega_N=1$ and
$\frac{\omega_{j-1}^2}{\omega_j}=\frac{1}{N}$ for every $j\in[N-1]$. Moreover, whenever
$\frac{1}{N}\leq\omega_j\leq1$, the recursion gives
$\frac{1}{N}\leq\omega_{j-1}\leq\omega_j$. Backward induction therefore proves
\eqref{eq:monotone-weight}.

Let $S_\omega:=\sum_{j=1}^{N-1}\omega_{j-1}$. By the AM-GM inequality,
$$
S_\omega
=
\sum_{j=1}^{N-1}\sqrt{\frac{\omega_j}{N}}
\leq
\frac12\left(
S_\omega-\omega_0+1+\frac{N-1}{N}
\right)
<
\frac12S_\omega+1.
$$
Thus $S_\omega<2$. Since $\omega_{N-1}=1$, the first inequality in
\eqref{eq:ncpl-weight-sums} follows. Finally, the recursion and
$\omega_{N-1}=\omega_N=1$ give
$$
\sum_{j=1}^N\frac{\omega_{j-1}^2}{\omega_j}
=
\frac{N-1}{N}+1
<2,
$$
which proves the second inequality.
\end{proof}

\subsection{Proof of Lemma~\ref{lem:ncpl-outer-scale-bounds}}
\label{app:proof-ncpl-outer-scale-bounds}

\begin{proof}[Proof of Lemma~\ref{lem:ncpl-outer-scale-bounds}]
Define
\[
\theta(s)
:=
\begin{cases}
0, & s\leq\frac12,\\[1mm]
\displaystyle
\exp\left(\frac12-\frac{1}{2(2s-1)^2}\right),
&s>\frac12.
\end{cases}
\]
Then $\theta(s)^2=\psi(s)$. Since the supports of $\psi(s)$ and
$\psi(-s)$ are disjoint,
\[
\rho(s)
=
\sqrt{\psi(s)+\psi(-s)}
=
\theta(s)+\theta(-s).
\]
The exponential cutoff is flat at $s=\frac12$, and direct
differentiation on $(\frac12,\infty)$ shows that $\theta$ and its first
two derivatives are bounded by numerical constants. Hence, $\theta$ is
smooth on $\R$, $\theta'$ is Lipschitz, and the same properties hold for
$\rho$. The identity above also gives
$\rho(s)=\rho'(s)=0$ whenever $|s|\leq\frac12$.
\end{proof}

{
\subsection{Proof of Lemma~\ref{lem:ncpl-lifted-target-bounds}}
\label{app:proof-ncpl-lifted-target-bounds}
}

\begin{proof}[Proof of Lemma~\ref{lem:ncpl-lifted-target-bounds}]
For this proof, write
\[
    \alpha^{+}(t):=5-\varphi(t),
    \qquad\text{and}\qquad
    \alpha^{-}(t):=5+\varphi(-t).
\]
The superscripts label the two regions and do not denote positive and
negative parts.
The range bound in Fact~\ref{lem:outer-activation-functions}~\textup{(iv)}
gives $0<\alpha^{\pm}<10$. Moreover,
$\sup_t|\varphi'(t)|=\sqrt e<2$ and
$\Lip(\varphi')=\sup_t|\varphi''(t)|=1$, which gives the asserted
derivative bound directly.

Using $\rho(s)^2=\psi(s)+\psi(-s)$, we obtain
$$
\widetilde h(s,t)
=
\psi(s)\alpha^{+}(t)+\psi(-s)\alpha^{-}(t).
$$
Since $\psi(s)$ and $\psi(-s)$ are active only when $s>\frac12$ and
$s<-\frac12$, respectively, this identity gives
\eqref{eq:lifted-target-active-form}.
\end{proof}

\subsection{Proof of Lemma~\ref{lem:ncpl-joint-smoothness}}
\label{app:proof-ncpl-joint-smoothness}

\begin{proof}[Proof of Lemma~\ref{lem:ncpl-joint-smoothness}]
We first rewrite \(C\) on the two regions where \(\rho(s)>0\). Define
\[
    \zeta_1(v):=1-p(v),
    \qquad\text{and}\qquad
    \zeta_2(u,v):=1-q(u)p(v),
\]
and
\[
    \alpha^{+}(t):=5-\varphi(t),
    \qquad\text{and}\qquad
    \alpha^{-}(t):=5+\varphi(-t).
\]
For \(\eta>0\), let
\[
    R_1(\eta;\y)
    :=
    \omega_0\eta^2
    \zeta_1\left(
        \frac{y_1}{\eta\sqrt{\omega_1}}
    \right),
\]
and, for \(j=2,\ldots,N\), let
\[
    R_j(\eta;\y)
    :=
    \omega_{j-1}\eta^2
    \zeta_2\left(
        \frac{y_{j-1}}{\eta\sqrt{\omega_{j-1}}},
        \frac{y_j}{\eta\sqrt{\omega_j}}
    \right).
\]
Set
\[
    R(\eta;\y):=\sum_{j=1}^N R_j(\eta;\y),
\]
and, for \(\nu\in\{+,-\}\), define
\[
    C^{\nu}(\eta,t;\y)
    :=
    -5\eta^2
    +\alpha^{\nu}(t)\eta^2
    q\left(\frac{y_N}{\eta}\right)
    -R(\eta;\y)
    -\frac12\norm{\y^-}^2.
\]
Expanding the definition of \(B\) gives
\begin{equation}
    \label{eq:C-two}
    C(s,t;\y)
    =
    \begin{cases}
        C^{+}(\rho(s),t;\y),
        & s>\frac12,\\
        C^{-}(\rho(s),t;\y),
        & s<-\frac12.
    \end{cases}
\end{equation}

Equation~\eqref{eq:C-two} reduces the analysis on the two outer regions to
uniform derivative bounds for \(C^\nu\). We first establish such
bounds for a single quadratic perspective, then apply them to the residual
sum \(R\) and the terminal term, and finally combine all terms in
\(C^\nu\). Afterward, we substitute \(\eta=\rho(s)\) and match the
resulting formulas with the central region. By
Lemma~\ref{lem:ncpl-outer-scale-bounds}, the relevant scale parameter
satisfies
\(
    0<\eta=\rho(s)\leq\norm{\rho}_{\infty}.
\)

The dependence on $\eta$ in both $R$ and the terminal term is
through expressions of the form $\eta^2Q(\boldsymbol u/\eta)$, where
$Q\in\{\zeta_1,\zeta_2,q\}$. We therefore consider
\[
    \widetilde Q(\eta,\boldsymbol u)
    :=
    \eta^2Q\left(\frac{\boldsymbol u}{\eta}\right),
    \qquad
    Q\in\{\zeta_1,\zeta_2,q\},
\]
and write \(\boldsymbol z:=\boldsymbol u/\eta\). Direct differentiation
gives
\[
\begin{aligned}
    &\nabla_{\eta}\widetilde Q
    =
    \eta\left(
        2Q(\boldsymbol z)
        -
        \langle\boldsymbol z,\nabla Q(\boldsymbol z)\rangle
    \right),\quad\text{and}\\
    &\nabla_{\boldsymbol u}\widetilde Q
    =
    \eta\nabla Q(\boldsymbol z),
\end{aligned}
\]
and, almost everywhere,
\[
\begin{aligned}
    &\nabla_{\eta\eta}\widetilde Q
    =
    2Q(\boldsymbol z)
    -
    2\langle\boldsymbol z,\nabla Q(\boldsymbol z)\rangle
    +
    \langle
        \boldsymbol z,
        \nabla^2Q(\boldsymbol z)\boldsymbol z
    \rangle,\\
    &\nabla_{\boldsymbol u\eta}\widetilde Q
    =
    \nabla Q(\boldsymbol z)
    -
    \nabla^2Q(\boldsymbol z)\boldsymbol z,\quad\text{and}\\
    &\nabla^2_{\boldsymbol u\boldsymbol u}\widetilde Q
    =
    \nabla^2Q(\boldsymbol z).
\end{aligned}
\]

For each \(Q\in\{\zeta_1,\zeta_2,q\}\), the function and its first
derivatives are uniformly bounded, as are its almost-everywhere second
derivatives. Moreover, whenever differentiation with respect to a coordinate
produces a nonzero term, that coordinate lies in a fixed transition
interval. Consequently,
\[
    \langle\boldsymbol z,\nabla Q(\boldsymbol z)\rangle,
    \qquad
    \langle
        \boldsymbol z,
        \nabla^2Q(\boldsymbol z)\boldsymbol z
    \rangle,
    \qquad\text{and}\qquad
    \nabla^2Q(\boldsymbol z)\boldsymbol z
\]
are uniformly bounded. Hence there exists a numerical constant
\(K_1\geq1\) such that, uniformly over all three choices of \(Q\),
\begin{equation}
    \label{eq:Qbound}
\begin{aligned}
    &|\widetilde Q|
    \leq K_1\eta^2,\\
    &\max\left\{
        |\nabla_\eta\widetilde Q|,
        \|{\nabla_{\boldsymbol u}\widetilde Q}\|
    \right\}
    \leq K_1\eta,\quad\text{and}\\
    &\max\left\{
        |\nabla_{\eta\eta}\widetilde Q|,
        \|{\nabla_{\boldsymbol u\eta}\widetilde Q}\|,
        \|{\nabla^2_{\boldsymbol u\boldsymbol u}\widetilde Q}\|
    \right\}
    \leq K_1.
\end{aligned}
\end{equation}

We now pass from the single-perspective estimate to the full residual chain.
Applying \eqref{eq:Qbound} and the chain rule to each \(R_j\), and
then using \eqref{eq:ncpl-weight-sums}, gives
\begin{equation}
    \label{eq:Rbound1}
\begin{aligned}
    &|R|
    <3K_1\eta^2,\\
    &|\nabla_\eta R|
    <3K_1\eta,
    \qquad
    \|{\nabla_{\y}R}\|^2
    <10K_1^2\eta^2,\\
    &|\nabla_{\eta\eta}R|
    <3K_1,
    \qquad\text{and}\quad
    \|{\nabla_{\y\eta}R}\|^2
    <10K_1^2.
\end{aligned}
\end{equation}
Furthermore, $R_1$ depends only on $y_1$, while, for
$j\geq2$, $R_j$ depends only on $y_{j-1}$ and $y_j$.
Thus \(\nabla^2_{\y\y}R\) is tridiagonal, and
\eqref{eq:monotone-weight} and \eqref{eq:Qbound} yield
\begin{equation}
    \label{eq:Ryy}
    \norm{\nabla^2_{\y\y}R}
    \leq
    \max_{j\in[N]}
    \sum_{k=1}^N
    \left|
        \nabla_{y_jy_k}R
    \right|
    \leq4K_1.
\end{equation}
Taking the gradient and Hessian here with respect to \((\eta,\y)\),
\eqref{eq:Rbound1} and \eqref{eq:Ryy} imply
\begin{equation}
    \label{eq:R-bound}
    \|{\nabla R}\|
    <5K_1\|\rho\|_\infty,
    \qquad\text{and}\qquad
    \|{\nabla^2R}\|
    <8K_1.
\end{equation}
Thus $R$ has derivative bounds independent of \(N\).

To control the full expression \(C^\nu\), we next consider its component that is dependent on $t$:
\[
    J^\nu(\eta,t;y_N)
    :=
    \alpha^\nu(t)\eta^2
    q\left(\frac{y_N}{\eta}\right).
\]
Lemma~\ref{lem:ncpl-lifted-target-bounds} gives a numerical constant
\(K_2\geq1\) such that, almost everywhere and uniformly in \(\nu\)
and \(t\),
\[
    \max\{|\alpha^\nu(t)|,
    |\nabla_t\alpha^\nu(t)|,
    |\nabla_{tt}\alpha^\nu(t)|\}
    \leq K_2.
\]
Similarly, applying \eqref{eq:Qbound} with $Q=q$ and using
$0<\eta\leq\norm{\rho}_\infty$, we obtain
\begin{equation}
    \label{eq:Tbound}
\begin{aligned}
    &|J^\nu|
    \leq K_1K_2\eta^2,\\
    &|\nabla_tJ^\nu|
    \leq K_1K_2\eta^2,\qquad\|{\nabla J^\nu}\|
    \leq
    K_1K_2
    \sqrt{2+\norm{\rho}_\infty^2}\,\eta,\quad\text{and}\\
    &\|{\nabla^2J^\nu}\|
    \leq
    K_1K_2
    \max\left\{
        2+\norm{\rho}_\infty,\,
        \norm{\rho}_\infty^2+2\norm{\rho}_\infty
    \right\}.
\end{aligned}
\end{equation}
Apart from $R$ and $J^\nu$, the remaining terms in $C^\nu$ are elementary:
\(-5\eta^2\) has constant second derivative, while
\(-\frac12\norm{\y^-}^2\) has gradient \(\y^-\), which is
\(1\)-Lipschitz. Combining these observations with
\eqref{eq:R-bound} and \eqref{eq:Tbound}, we obtain a numerical constant \(K\), independent of
\(N\), such that almost everywhere
\begin{equation}
    \norm{\nabla^2C^\nu(\eta,t;\y)}
    \leq K.\label{eq:Cnu-bound-1}
\end{equation}
Moreover, uniformly in \(\nu\), \(t\), \(\y\), and the relevant
range of \(\eta\), we have
\begin{equation}
\begin{aligned}
    &\!\left|
        C^\nu(\eta,t;\y)
        +\frac12\norm{\y^-}^2
    \right|
    \leq K\eta^2,\\
    &|\nabla_\eta C^\nu(\eta,t;\y)|
    \leq K\eta,\quad|\nabla_tC^\nu(\eta,t;\y)|
    \leq K\eta^2,\quad\text{and}\quad\|{
        \nabla_{\y}C^\nu(\eta,t;\y)-\y^-
    }\|
    \leq K\eta.
\end{aligned}\label{eq:Cnu-bound-2}
\end{equation}

We now return to the original block \(C\) by substituting
\(\eta=\rho(s)\) in \eqref{eq:C-two}. Set
\(\boldsymbol v:=(t,\y)\). On either outer region
\(|s|>\frac12\), the chain rule gives, almost everywhere,
\[
\begin{aligned}
    &\nabla_{ss}C
    =
    \nabla_{\eta\eta}C^\nu(\rho')^2
    +
    \nabla_\eta C^\nu\rho'',\\
    &\nabla_{\boldsymbol v s}C
    =
    \rho'\nabla_{\boldsymbol v\eta}C^\nu,\quad\text{and}\\
    &\nabla^2_{\boldsymbol v\boldsymbol v}C
    =
    \nabla^2_{\boldsymbol v\boldsymbol v}C^\nu,
\end{aligned}
\]
where \(\nu=+\) for \(s>\frac12\), \(\nu=-\) for
\(s<-\frac12\), and all derivatives of \(C^\nu\) are evaluated at
\((\rho(s),t;\y)\). By \eqref{eq:Cnu-bound-1},
\eqref{eq:Cnu-bound-2}, and
Lemma~\ref{lem:ncpl-outer-scale-bounds}, using
\(
    |\rho''|\leq\Lip(\rho')
\)
almost everywhere, we obtain
\[
\begin{aligned}
    &|\nabla_{ss}C|
    \leq
    K\norm{\rho'}_\infty^2
    +
    K\norm{\rho}_\infty\Lip(\rho'),\\
    &\|{\nabla_{\boldsymbol v s}C}\|
    \leq
    K\norm{\rho'}_\infty,\quad\text{and}\\
    &\|{\nabla^2_{\boldsymbol v\boldsymbol v}C}\|
    \leq K.
\end{aligned}
\]
Consequently,
\begin{equation}
    \label{eq:C-outer-Hessian-bound}
    \norm{\nabla^2C}
    \leq
    K\left(
        1+\norm{\rho'}_\infty
        +\norm{\rho'}_\infty^2
        +\norm{\rho}_\infty\Lip(\rho')
    \right)
\end{equation}
almost everywhere on both outer regions. This is a numerical bound
independent of \(N\).

This proves the required Hessian bound away from the interfaces
\(s=\pm\frac12\). On the central region, the definition of \(B\) gives
\[
    C(s,t;\y)
    =
    -\frac12\norm{\y^-}^2
    \qquad
    \text{for }
    |s|\leq\frac12,
\]
whose gradient is \(1\)-Lipschitz. Since
\(\rho(s)=\rho'(s)=0\) at \(s=\pm\frac12\), the value and
first derivative estimates for \(C^\nu\) show that the outer
formulas match the central formula continuously. In particular, when
\(s=\pm\frac12\),
\[
    \nabla C(s,t;\y)
    =
    (0,0,\y^-).
\]

Let \(L_C\) be the larger of \(1\) and the numerical bound in
\eqref{eq:C-outer-Hessian-bound}. By
Lemmas~\ref{lem:ncpl-gate-bounds},
\ref{lem:ncpl-outer-scale-bounds}, and
\ref{lem:ncpl-lifted-target-bounds}, the gradient \(\nabla C\) is locally
Lipschitz on each outer region. Since both outer regions are convex, the
almost-everywhere Hessian bound implies, by the standard mollification
argument \citep[Section~4.2.3, Theorem~5]{evansgariepy1992}, that
\(\nabla C\) is \(L_C\)-Lipschitz on each of them. By the continuity of
\(\nabla C\) across \(s=\pm\frac12\), the same estimate extends to their
closures. On the central region, the same estimate follows directly from
\[
    \nabla C(s,t;\y)=(0,0,\y^-),
    \qquad |s|\leq\frac12,
\]
since the map \(\y\mapsto\y^-\) is \(1\)-Lipschitz and \(L_C\geq1\).
Finally, split any line segment at its intersections with
\(s=\pm\frac12\). Applying the preceding estimate to each resulting piece
and summing their lengths proves that \(\nabla C\) is globally
\(L_C\)-Lipschitz. Hence \(C\) is globally \(L_C\)-smooth, where \(L_C\)
is a numerical constant independent of \(N\).
\end{proof}
\end{document}